\documentclass[12pt]{amsart}

\usepackage{amsfonts}
\usepackage{amsmath}
\usepackage{amssymb}
\usepackage{amstext}
\usepackage{amsthm}
\usepackage[all]{xy}

\usepackage{color}

\newcommand{\R}{\mathbb{R}}

\newcommand{\Z}{\mathbb{Z}}

\newcommand{\scr}{\mathcal}
\newcommand{\wt}{\widetilde}

\newtheorem{thm}{Theorem}[section]
\newtheorem{lemma}[thm]{Lemma}
\newtheorem{cor}[thm]{Corollary}
\newtheorem{prop}[thm]{Proposition}

\theoremstyle{definition}
\newtheorem{dfn}[thm]{Definition}
\newtheorem{rem}[thm]{Remark}
\newtheorem{ex}[thm]{Example}

\definecolor{KS}{RGB}{50,100,150}
\definecolor{BC}{RGB}{150,100,0}

\author{Baptiste Chantraine}
\address{Nantes Université, CNRS, Laboratoire de Mathématiques Jean Leray, LMJL,
	UMR 6629, F-44000 Nantes, France}
\email{baptiste.chantraine@univ-nantes.fr}
\author{Kevin Sackel}
\address{Department of Mathematics and Statistics, University of Massachusetts, Amherst, MA 01003, USA}
\email{ksackel@umass.edu}
\title{Products of locally conformal symplectic manifolds}
\begin{document}

\maketitle

\begin{abstract}
Given two locally conformal symplectic (LCS) structures on manifolds $M_1$ and $M_2$, we construct a natural $\R^+$-torsor of locally conformal symplectic structures on a certain covering space $M_1 \boxplus M_2$ of $M_1 \times M_2$. As the smooth construction of $M_1 \boxplus M_2$ is natural from the perspective of flat line bundles, we use this language to phrase the LCS theory. This construction shares many properties with, and in a sense generalizes, the standard symplectic product. Notably, for a Hamiltonian isotopy $\phi_t$ of an LCS manifold $M$, there is an associated Lagrangian embedding $\Gamma(\phi_1) \colon M \hookrightarrow M \boxplus M$, in which certain fixed points of $\phi_1$ are in bijection with intersection points of $\Gamma(\phi_1)$ with the diagonal $\Delta = \Gamma(\mathrm{id})$. Using a Lagrangian intersection of result of the first author and E. Murphy, we may conclude that if $\phi_t$ is a $C^0$-small Hamiltonian isotopy, then the number of fixed points of $\phi_1$ is bounded below by the rank of the Novikov theory associated to the Lee class of the LCS structure on $M$. Finally, we end the paper by constructing the suspension of a Lagrangian submanifold along a Hamiltonian isotopy in the LCS theory, again generalizing the symplectic setting.
\end{abstract}

\section{Introduction.}
\label{sec:introduction}

Locally conformal symplectic (LCS) manifolds are manifolds on which we can perform Hamiltonian dynamics. There are many equivalent definitions of LCS manifolds; for the purposes of this paper, we use a formulation in terms of flat oriented line bundles. Background details of flat line bundles which are necessary to understand the following definition will be recalled in Section \ref{sec:flat-line-bundles}. An equivalent definition of LCS manifolds which is quite common in the literature and which is useful for computations will be recalled in Definition \ref{dfn:LCS_gauge}.

\begin{dfn} \label{dfn:LCS}
	An \textbf{LCS manifold} consists of the data of
	\begin{itemize}
		\item a smooth manifold $M$,
		\item a flat oriented line bundle $\pi \colon (E,\nabla) \rightarrow M$, and
		\item an $E$-valued differential $2$-form $\omega \in \Omega^2(M;E)$ which is
		\begin{itemize}
			\item \textbf{closed}, i.e. $d_\nabla \omega = 0$, and
			\item \textbf{non-degenerate}, i.e. the interior product map $\iota_{\bullet}\omega \colon TM \rightarrow T^*M \otimes E$ is an isomorphism of vector bundles.
		\end{itemize}
	\end{itemize}
	An \textbf{LCS map} $(\phi,\Phi) \colon (M_1,E_1,\nabla_1,\omega_1) \rightarrow (M_2,E_2,\nabla_2,\omega_2)$ is a morphism of flat oriented line bundles
	$$\xymatrix{(E_1,\nabla_1) \ar[r]^{\Phi} \ar[d] & (E_2,\nabla_2) \ar[d] \\ M_1 \ar[r]^{\phi} & M_2}$$
	such that $\Phi^*\omega' = \omega$.
\end{dfn}

\begin{rem}
	Though we work with this definition, it obfuscates the terminology `locally conformal symplectic'. To recover the terminology, note that any flat line bundle is locally trivial, i.e. around any point $p$, there is an open neighborhood $U$ and an isomorphism of flat oriented line bundles $\Psi \colon (E|_U,\nabla) \cong (\underline{\R}|_U,d)$, where $\underline{\R}$ is the trivial line bundle. Under such an identification, $\omega|_U$ corresponds to a symplectic form. In this manner, an LCS manifold comes with a collection of local symplectic forms. Furthermore, $\Psi$ is determined up to multiplication by a positive constant (c.f. Proposition \ref{prop:flat-line-bundle-facts} below), so $\omega$ is only defined up to multiplication by a positive constant: this explains the `conformal' part of the terminology.
\end{rem}

\begin{rem} \label{rem:base_only}
	If $(\phi,\Phi)$ and $(\phi,\Phi')$ are LCS maps between the same LCS manifolds, then $\Phi = \Phi'$ (as we will discuss in Section \ref{sec:LCS}). Hence, we may refer to the LCS map by $\phi$ without reference to $\Phi$. However, not every smooth map from $M_1$ to $M_2$ is the base of an LCS map.
\end{rem}

It appears to be an unfortunate fact of LCS geometry that if $M_1$ and $M_2$ are given LCS structures, then there is no obvious way to construct an LCS structure on the smooth manifold $M_1 \times M_2$. The main goal of this paper is the proposal of a notion of product of LCS manifolds that appears well-suited to generalizing typical applications of the standard symplectic product to the LCS setting.

Although we are interested in the LCS setting, underlying our LCS product is a smooth version for manifolds with chosen first cohomology classes, which is part of the data of an LCS manifold. Indeed, recall that to any flat line bundle $\pi \colon (E,\nabla) \rightarrow M$, there is a holonomy class $\mathrm{hol}(\nabla) \in H^1(M;\R^*)$. Since on an LCS manifold, the line bundles are oriented, the holonomy class lives in $H^1(M;\R^+)$. Applying the natural logarithm yields a class in $H^1(M;\R)$.

\begin{dfn}
	For an LCS manifold $(M,E,\nabla,\omega)$, the cohomology class
	$$\ln(\mathrm{hol}(\nabla)) \in H^1(M;\R)$$
	is called the \textbf{Lee class}.
\end{dfn}

Smoothly, the product we construct is as follows, the content of which is an elementary byproduct of the theory of covering spaces.
\begin{thm} \label{thm:main_smooth}
	Suppose $M_1$ and $M_2$ are two smooth manifolds with specified cohomology class $\mathfrak{l}_1 \in H^1(M_1;\R)$ and $\mathfrak{l}_2 \in H^1(M_2;\R)$. Then there exists a smooth manifold $M_1 \boxplus M_2$ (depending upon $\mathfrak{l}_1$ and $\mathfrak{l}_2$) together with a covering map
	$$r \colon M_1 \boxplus M_2 \rightarrow M_1 \times M_2,$$
	canonical up to deck transformations and satisfying the following properties:
	\begin{enumerate}
		\item We have $(\pi_1 \circ r)^*\mathfrak{l}_1 = (\pi_2 \circ r)^*\mathfrak{l}_2 \in H^1(M_1 \boxplus M_2;\R)$. In fact, $r$ is the minimal such covering of $M_1 \times M_2$ with this property.
		\item Given any smooth map $\phi \colon M_1 \rightarrow M_2$ such that $\phi^*\mathfrak{l}_2 = \mathfrak{l}_1$, there is a smooth embedding
		$$\Gamma(\phi) \colon M_1 \rightarrow M_1 \boxplus M_2,$$
		canonical up to the deck group action of $r$, which lifts the usual graph
		$$\Gamma_0(\phi) \colon M_1 \rightarrow M_1 \times M_2,$$
		i.e. with $\Gamma_0(\phi) = r \circ \Gamma(\phi)$.
		\item If $(M_1,\mathfrak{l}_1) = (M_2,\mathfrak{l}_2)$, in which case we write simply $(M,\mathfrak{l})$ for the pair, then the covering map $r$ is canonical. In such a case, given any smooth family of diffeomorphisms $\phi_t \in \mathrm{Diff}(M)$ starting at the identity, the lifts $\Gamma(\phi_t)$ may be chosen canonically, and there is a canonical smooth family of diffeomorphisms $\psi_t \in \mathrm{Diff}(M \boxplus M)$ lifting $\mathrm{Id} \times \phi_t \in \mathrm{Diff}(M \times M)$ such that $\Gamma(\phi_t) = \psi_t \circ \Gamma(\mathrm{Id})$. Furthermore, the intersection points $\Gamma(\mathrm{Id}) \cap \Gamma(\phi_1)$ are in bijective correspondence with those fixed points $x \in \mathrm{Fix}(\phi_1)$ such that
		$$\langle \mathfrak{l}, [\phi_t(x)] \rangle = 0,$$
		and such that an intersection point in $\Gamma(\mathrm{Id}) \cap \Gamma(\phi_1)$ is transverse if and only if the corresponding fixed point of $\phi_1$ is non-degenerate.
	\end{enumerate}
\end{thm}

Our LCS result is then the construction of an LCS structure on this product with nice properties. The following may be considered the main theorem of this paper.
\begin{thm} \label{thm:main_LCS}
	Let $\scr{M}_1 = (M_1,E_1,\nabla_1,\omega_1)$ and $\scr{M}_2 = (M_2,E_2,\nabla_2,\omega_2)$ be two connected LCS manifolds with Lee classes $\mathfrak{l}_1 \in H^1(M_1;\R)$ and $\mathfrak{l}_2 \in H^1(M_2;\R)$, respectively. Let
	$$r \colon \scr{M}_1 \boxplus \scr{M}_2 \rightarrow M_1 \times M_2$$
	denote the product construction of Theorem \ref{thm:main_smooth} with respect to the Lee classes $\mathfrak{l}_1$ and $\mathfrak{l}_2$. Then:
	\begin{enumerate}
		\item There is a natural family of LCS structures $\Xi$ on $\scr{M}_1 \boxplus \scr{M}_2$, where $\Xi$ is an $\R^+$-torsor.
		\item The deck group action of $r$ intertwines the LCS structures freely, i.e. if $\Gamma$ is the deck group, then there is an injective group homomorphism $\rho \colon \Gamma \rightarrow \R$ such that for each $g \in \Gamma$ and each $\Omega \in \Xi$, we have that the corresponding $\phi_g \in \mathrm{Diff}(\scr{M}_1\boxplus \scr{M}_2)$ is an LCS symplectomorphism from $\Omega$ to $e^{\rho(g)}\cdot \Omega$.
		\item The Lee class of any LCS structure in $\Xi$ is $(\pi_1 \circ r)^*\mathfrak{l}_1 = (\pi_2 \circ r)^*\mathfrak{l}_2$. 
		\item Suppose we fix a lift $\Gamma(\phi)$ for a smooth map $\phi \colon M_1 \rightarrow M_2$. Then $\phi$ is an LCS map if and only if there exists an LCS structure $\Omega_{\Gamma(\phi)} \in \Xi$ such that $\Gamma(\phi)$ is an isotropic embedding, in which case this LCS structure is uniquely determined.
		\item If $\phi_t$ is a smooth family of LCS automorphisms of $\scr{M}$, then each corresponding diffeomorphism $\psi_t$ of $\scr{M} \boxplus \scr{M}$ is an LCS symplectomorphism from $\Omega_{\Gamma(\mathrm{Id})}$ to $\Omega_{\Gamma(\phi_t)}$. Explicitly, if $\phi_t$ is generated by an LCS vector field $X_t$ with expansion coefficient $\mu_t$ (see Definition \ref{dfn:LCS_vf}), then
		$$\Omega_{\Gamma(\phi_t)} = e^{\int_0^t \mu_{\tau}d\tau} \cdot \Omega_{\Gamma(\mathrm{Id})}.$$
		\item If $\phi_t$ is an LCS-Hamiltonian isotopy of $\scr{M}$ (see Definition \ref{dfn:LCS-Hamiltonian}), then $\Omega_{\Gamma(\phi_t)}=\Omega_{\Gamma(\operatorname{Id})}$ and the corresponding $\psi_t$ is an LCS-Hamiltonian isotopy of this LCS structure.
	\end{enumerate}
\end{thm}

\begin{rem}
	In our construction, each of the LCS structures $\Omega \in \Xi$ is given as a triple $(r_1^*E_1,r_1^*\nabla_1, \omega_{\Omega})$. However, it is important to note that even though $\Xi$ is an $\R^+$-torsor, we have in general that
	$$\omega_{e^c \cdot \Omega} \neq e^c \cdot \omega_{\Omega}.$$ 
\end{rem}
\color{black}

There is a subclass of LCS manifolds which play particularly nicely with the product. Roughly speaking, an LCS manifold $(M,E,\nabla,\omega)$ is said to be exact if there is some $\lambda \in \Omega^1(M;E)$ with $\omega = d_\nabla \lambda$. Strictly speaking, the data of $\lambda$ should be included as part of the data of an \textbf{exact LCS (ELCS) manifold}; we refer to Definition \ref{dfn:ELCS} for the correct statement.
\color{black}

\begin{thm}\label{thm:exact}
	Let $\scr{M}_1$, $\scr{M}_2$ and $\Xi$ be as in Theorem \ref{thm:main_LCS}. If either $\scr{M}_1$ or $\scr{M}_2$ is ELCS, then any two LCS structures in $\Xi$ are LCS-symplectomorphic. If both $\scr{M}_1$ and $\scr{M}_2$ are ELCS, then $\scr{M}_1\boxplus \scr{M}_2$ is also ELCS for any of the LCS structures in $\Xi$.
\end{thm}

As a proof of concept of the utility of this product, we generalize an argument of Laudenbach and Sikorav \cite{zbMATH03951633} to the LCS setting. In their work, they prove that if $L$ is a closed manifold and $\phi_t$ is a Hamiltonian isotopy on a cotangent bundle $T^*L$ such that $\phi_1(L)$ is transverse to $L$, then
$$|\phi_1(L) \cap L| \geq \mathrm{rk}(H_*(L)).$$
They then use this to prove that if $\psi_t$ is a $C^0$-small Hamiltonian isotopy of a closed manifold $(M,\omega)$ such that $\psi_1$ has non-degenerate fixed points, then
$$|\mathrm{Fix}(\psi_1)| \geq \mathrm{rk}(H_*(M)).$$
The Hamiltonian fixed point theorem follows from the Lagrangian intersection theorem since the graphs $\Gamma(\psi_t) \subset M \times \overline{M}$ yield a Hamiltonian isotopy of the diagonal $\Delta \subset M \times \overline{M}$, contained in a neighborhood of the diagonal symplectomorphic to $T^*\Delta$.

In the LCS setting there is still a notion of Hamiltonian isotopy, generated by flowing along LCS-Hamiltonian vector fields. Namely, if $H \in \Gamma(E)$ is a section of $E$, then there is a vector field $X_H$ uniquely defined by the equation
$$\iota_{X_H}\omega = -\nabla H.$$
Locally, any flat oriented line bundle is isomorphic to the trivial one, so that $X_H$ may be considered locally Hamiltonian with respect to a symplectic form in the usual sense. See also Definition \ref{dfn:LCS-Hamiltonian}, which includes a description of the full LCS map on the bundle, not just the diffeomorphism on the base, c.f. Remark \ref{rem:base_only}.

\begin{rem}
	Despite this local similarity to symplectic geometry, the global behaviour of LCS Hamiltonian flows on compact manifolds differs significantly from the standard symplectic setting; we refer to Allais--Arnaud \cite{allais2022dynamics} for a list of different behaviours such Hamiltonian flows can have.
\end{rem}

The first author together with E. Murphy \cite{zbMATH07103007} proved a Lagrangian intersection result analogous to the one described in the Laudenbach--Sikorav setting, but on twisted cotangent bundles (see Example \ref{excotan} for a definition). Instead of using usual homology, they phrase their result in terms of the so-called Novikov homology \cite{zbMATH03796853}, which takes as input a manifold $M$ together with a cohomology class in $H^1(M;\R)$. Using the product of Theorem \ref{thm:main_LCS}, we may now use the Laudenbach--Sikorav argument to obtain a Hamiltonian fixed point theorem in the LCS setting.

\begin{thm} \label{thm:fix}
	Let $\scr{M}$ be a closed LCS manifold with Lee class $\mathfrak{l} \in H^1(M;\R)$. If $\phi_t$ is a $C^0$-small LCS-Hamiltonian isotopy such that $\phi_1$ has non-degenerate fixed points, then
	$$\#\operatorname{Fix}(\phi_1)\geq \mathrm{rk}(\mathrm{Nov}(\mathfrak{l})).$$
\end{thm}

\begin{proof}
	By Theorem \ref{thm:main_LCS}, we have that $\Gamma(\mathrm{Id}) \colon M \rightarrow \scr{M} \boxplus \scr{M}$ is a Lagrangian embedding. The Lee class pulls back to $M$ as
	$$\Gamma(\mathrm{Id})^*(\pi_1 \circ r)^*\mathfrak{l} = (\pi_1 \circ r \circ \Gamma(\mathrm{Id}))^*\mathfrak{l} = (\pi_1 \circ \Gamma_0(\mathrm{Id}))^* = \mathrm{Id}^*\mathfrak{l} = \mathfrak{l}.$$
	By the Weinstein Neighborhood Theorem in the LCS setting (see \cite[Theorem 1.1]{OTIMAN20171} or \cite[Theorem 4.2]{zbMATH06631798}), a neighborhood of a Lagrangian is determined up to LCS symplectomorphism by the restriction of the Lee class to the Lagrangian. In particular, a neighborhood of $\Gamma(\mathrm{Id})$ is LCS-symplectomorphic to a neighborhood of the zero section in a twisted cotangent bundle.
	
	Now let $\psi_t \in \mathrm{Diff}(\scr{M} \boxplus \scr{M})$ denote the lift of $\mathrm{Id} \times \phi_t \in \mathrm{Diff}(M \times M)$ as in Theorem \ref{thm:main_smooth}. We note that since $r$ is a covering space, since $\phi_t$ is $C^0$-small, we also have that $\psi_t$ is $C^0$-small. Furthermore, by Theorem \ref{thm:main_LCS}, we have that $\psi_t$ is Hamiltonian, and so $\Gamma(\phi_t) = \psi_t \circ \Gamma(\mathrm{Id})$ is a Hamiltonian deformation of $\Gamma(\mathrm{Id})$ in $U$. Since the fixed points of $\phi_1$ are non-degenerate, we have that $\Gamma(\mathrm{Id}) \cap \Gamma(\phi_1)$ is transverse. Given the identification of $U$ with a neighborhood of the zero section in a twisted cotangent bundle, applying the Lagrangian fixed point result of the first author and Murphy \cite{zbMATH07103007}, we have
	$$\# \mathrm{Fix}(\phi_1) \geq \# (\Gamma(\mathrm{Id}) \cap \Gamma(\phi_1)) \geq \mathrm{rk}(\mathrm{Nov}(\mathfrak{l})).$$
\end{proof}

\begin{rem}
	Under the same hypotheses as Theorem \ref{thm:fix}, we may use a slightly stronger version of the Lagrangian intersection result \cite{zbMATH07103007} to prove that
	$$\#\operatorname{Fix}_0(\phi_1;\phi_t)\geq\mathrm{StabMorse}_{\mathfrak{l}}(M),$$
	where the left-hand side is the number of fixed points $x$ with $\langle \mathfrak{l}, [\phi_t(x)] \rangle = 0$, and the right-hand side is a (twisted) stable Morse number, which is at least the rank of $\mathrm{Nov}(\mathfrak{l})$.
\end{rem}

\textbf{For the rest of the paper, we will assume that all of our manifolds are connected.} Indeed, it is easy to see that the general case of Theorem \ref{thm:main_smooth} follows from the connected case. Generalizing Theorem \ref{thm:main_LCS} to the disconnected case is not much more difficult -- the only change is that $\Xi$ is an $(\R^+)^S$-torsor, where $S = \pi_0(M_1) \times \pi_0(M_2)$. Similarly, the disconnected case of Theorem \ref{thm:fix} follows easily from the connected case, and hence we may assume $\scr{M}$ is connected in the proof. The exposition is made particularly clean by considering the connected case alone, so we follow this convention.

\begin{rem} 
	Most of our geometric constructions hold for arbitrary flat line bundles, not just oriented ones. Both LCS geometry and Hamiltonian dynamics in this more general setting have some interest, and adapting our definitions to these cases is mainly obvious (when possible), but the reader should be advised that Definition \ref{dfn:LCS_gauge} is no longer available, i.e. one cannot simply fix a gauge representative of the connection and work with usual differential forms. More importantly, it is not clear how the Lagrangian intersection result of the first author and E. Murphy \cite{zbMATH07103007} adapts in this situation, and hence how Theorem \ref{thm:fix} should be modified to account for this situation, though surely it would involve some $\mathbb{Z}_2$-equivariance in Novikov theory.
\end{rem}

\subsection{Acknowledgements}
This collaboration started during the 2021 meeting ``Conformal symplectic structures, contact topology, and foliations'' of the American Institute of Mathematics (AIM). Both the authors wish to thank AIM together with the organisers of the meeting for creating this possiblity. The authors thanks M\'elanie Bertelson, Yasha Eliashberg and Emmy Murphy for preliminary discussions on the subject. The first author is partially supported by the ANR projects COSY (ANR-21-CE40-0002) and COSYDY (ANR-CE40-0014). Initial work on this project began while the second author was supported by the National Science Foundation under grant DMS-1547145.

\section{Flat line bundles} \label{sec:flat-line-bundles}

In this section, we recall details about flat line bundles and their associated cochain complexes of differential forms. Everything in this section is standard material; see for example \cite[Chapter 12]{zbMATH05652257} and \cite[Section 7]{zbMATH03782042}.

\subsection{Flat line bundles and the twisted de Rham complex}
Suppose $\pi \colon E \rightarrow M$ is a line bundle. We denote by $\Omega^{\bullet}(M;E) := \Gamma(\Lambda^{\bullet}T^*M \otimes E)$ the smooth $E$-valued differential forms. A connection is a linear map
$$\nabla \colon \Omega^{0}(M;E) \rightarrow \Omega^{1}(M;E)$$
such that for any $s \in \Gamma(E)$ and $f \in C^{\infty}(M)$, we have
$$\nabla (fs) = df \otimes s + f\nabla s.$$
This extends to a linear map
$$d_\nabla \colon \Omega^{\bullet}(M;E) \rightarrow \Omega^{\bullet+1}(M;E)$$
in a unique manner so that it satisfies the Leibniz rule: for any differential form $\nu \in \Omega^{\bullet}(M)$ and any section $s \in \Gamma(E)$, we have
$$d_\nabla(\nu \otimes s) = d\nu \otimes s + \nu \wedge d_\nabla s.$$
The connection is called flat if $d_\nabla^2 = 0$, in which case we obtain a cochain complex $(\Omega^{\bullet}(M;E),d_\nabla)$, which we shall call the \textbf{twisted de Rham complex}. The trivial flat line bundle is $(\underline{\R},d)$, for which the twisted de Rham complex is just the usual de Rham complex.

\color{black}

If $(E_1,\nabla_1)$ and $(E_2,\nabla_2)$ are two flat line bundles over the same manifold $M$, then $E_1 \otimes E_2$ comes with a canonical flat connection so that the wedge product induces a map of cochain complexes
$$\wedge \colon \Omega^{\bullet}(M;E_1) \otimes \Omega^{\bullet}(M;E_2) \rightarrow \Omega^{\bullet}(M;E_1 \otimes E_2).$$
Taking either $(E_1,\nabla_1)$ or $(E_2,\nabla_2)$ as the trivial flat line bundle, we obtain that $(\Omega^{\bullet}(M;E),d_\nabla)$ is a differential graded bimodule over the usual de Rham differential graded algebra.

Suppose $\pi_1 \colon E_1 \rightarrow M_1$ and $\pi_2 \colon E_2 \rightarrow M_2$ are line bundles, and suppose we have a map of line bundles $(\phi,\Phi)$, by which we mean maps $\phi$ and $\Phi$ such that the following diagram commutes with $\Phi$ fibrewise linear:
$$\xymatrix{E_1 \ar[r]^{\Phi} \ar[d]_{\pi_1} & E_2 \ar[d]^{\pi_2} \\ M_1 \ar[r]^{\phi} & M_2}.$$
We will often write $\Phi \colon E_1 \rightarrow E_2$ lying over $\phi$. The data of $\Phi$ is equivalently given by a morphism $\Phi \colon E_1 \rightarrow \phi^*E_2$ (lying over the identity on $M_1$), which we also denote by $\Phi$ by abuse of notation. In the event that $\Phi$ is a fibrewise isomorphism (i.e. inducing an isomorphism of line bundles $E_1 \cong \phi^*E_2$ over $M_1$), then there is an induced map for any $q \in M_1$ of the form:
$$\Phi_q^* \colon \Lambda^{\bullet}T^*_{\phi(q)}M_2 \otimes (E_2)_{\phi(q)} \rightarrow \Lambda^{\bullet}T^*_{q}M_1 \otimes (E_1)_q$$
given by $\phi^* \otimes \Phi_q^{-1}$, where $\Phi_q \colon (E_1)_q \rightarrow (E_2)_{\phi(q)}$ is the corresponding isomorphism of fibers. These glue to give a pull-back map
$$\Phi^* \colon \Omega^{\bullet}(M_2;E_2) \rightarrow \Omega^{\bullet}(M_1;E_1).$$
Any connection $\nabla$ on $E_2$ then induces a connection $\Phi^*\nabla$ on $E_1$ by requiring that
$$(\Phi^*\nabla)(\Phi^*s) = \Phi^*(\nabla s)$$
for all $s \in \Gamma(E_2)$. If $\nabla$ is flat, then so is $\Phi^*\nabla$. If now our line bundles come with connections $\nabla_1$ and $\nabla_2$, then we shall regard the pair $(\phi,\Phi)$ as a \textbf{morphism of flat line bundles} if $\Phi$ is a fiberwise isomorphism and $\nabla_1 = \Phi^*\nabla_2$. In such a case, the map $\Phi^*$ on the twisted de Rham complexes is actually a map of cochain complexes
$$\Phi^* \colon (\Omega^{\bullet}(M_2;E_2),d_{\nabla_2}) \rightarrow (\Omega^{\bullet}(M_1;E_1),d_{\nabla_1}).$$

As a special case, suppose we fix any smooth map $\phi \colon N \rightarrow M$ and line bundle $\pi \colon E \rightarrow M$. Then we have a canonical line bundle over $N$ given by pullback, $\pi' \colon \phi^*E \rightarrow N$. Notice that there is a canonical bundle map over $\phi$ in this case, $\Phi_{\mathrm{can}} \colon \phi^*E \rightarrow E$ lying over $\phi$, which is automatically a fibrewise isomorphism. We hence have $(\phi,\Phi_{\mathrm{can}})$ is a morphism of flat line bundles $(\phi^*E,\Phi_{\mathrm{can}}^*\nabla) \rightarrow (E,\nabla)$ lying over $\phi$. By abuse of notation, since in this case $\Phi_{\mathrm{can}}$ is determined by $\phi$, we will just write
$$\phi^*\nabla := \Phi_{\mathrm{can}}^*\nabla$$
and
$$\phi^* := \Phi_{\mathrm{can}}^* \colon (\Omega^{\bullet}(M;E),d_\nabla) \rightarrow (\Omega^{\bullet}(N;\phi^*E),d_{\phi^*\nabla}).$$

Suppose $p,q \in M$, and suppose $\gamma$ is a path from $p$ to $q$ in $M$. For any element $v \in E_p$, there exists a unique section $s \in \Gamma(\gamma^*E)$ such that $s(0) = v$ and $(\gamma^*\nabla)s = 0$. By assigning to $v$ the value of $s(1) \in E_q$, we obtain a linear isomorphism
$$F_{\gamma} \colon E_p \xrightarrow{\sim} E_q,$$
called the parallel transport map. Assuming $\nabla$ is flat, the parallel transport depends only upon the homotopy class of $\gamma$ relative to endpoints. If $p = q$, then $F_{\gamma} \in \mathrm{Aut}(E_p) \cong \R^*$, so we obtain a homomorphism
$$\pi_1(M,p) \rightarrow \R^*.$$
This in turn may be thought of as a cohomology class $\mathrm{hol}(\nabla) \in H^1(M;\R^*)$, which is called the holonomy class. It is natural under morphisms of flat bundles; such a morphism $(\phi,\Phi)$ has $\phi^*\mathrm{hol}(\nabla_2) = \mathrm{hol}(\nabla_1)$. For oriented line bundles, $\mathrm{hol}(\nabla) \in H^1(M;\R^+)$. The holonomy class essentially specifies all of the data of the oriented flat line bundle up to isomorphism, in the sense of the following proposition.

\begin{prop} \label{prop:flat-line-bundle-facts}
	Suppose $\rho \in H^1(M;\R)$ is a fixed cohomology class. Then there exists a flat oriented line bundle $\pi \colon (E,\nabla) \rightarrow M$ such that
	$$e^{\rho} = \mathrm{hol}({\nabla}) \in H^1(M;\R^+).$$
	Any two flat oriented line bundles over $M$ with the same holonomy class are isomorphic, and the set of isomorphisms covering the identity on $M$ is naturally an $\R^+$-torsor, where the $\R^+$-action is given by multiplying the isomorphism by a positive constant.
\end{prop}

\begin{rem}
	We recall our convention that $M$ is connected. Otherwise, the set of isomorphisms in the above proposition would be an $(\R^+)^{\pi_0(M)}$-torsor, corresponding to multiplying the isomorphism by a positive locally constant function instead of just a positive constant.
\end{rem}

\begin{rem}
This proposition tells us that the twisted de Rham complex is in some sense an invariant of $\rho \in H^1(M;\R)$, up to scaling the complex by positive constants. Indeed, the proposition tells us that for any $\rho$, we may build a flat line bundle to obtain the twisted de Rham complex, and any two such complexes are isomorphic via an isomorphism which is canonical up to scale by a global positive constant.
\end{rem}

If $\phi \colon M \rightarrow N$ is a smooth map, and if $\mathcal{E} = (E,\nabla_E)$ and $\mathcal{F} = (F,\nabla_F)$ are flat oriented line bundles over $M$ and $N$, respectively, then we may denote by
$$\mathrm{Iso}_{\phi}(\mathcal{E},\mathcal{F})$$
the set of flat oriented line bundle isomorphisms lying over $\phi$. Notice that assuming $M$ is connected, then this is either an $\R^+$-torsor, if $\phi^*\mathrm{hol}(\nabla_F) = \mathrm{hol}(\nabla_E)$, or empty otherwise. We will typically write simply $\mathrm{Iso}(\scr{E},\scr{F})$ if $\phi = \mathrm{id}$.

\begin{prop}\label{prop:restrict_torsor}
	Suppose we have a commutative diagram of smooth maps between connected manifolds (possibly with boundary)
	$$\xymatrix{N \ar[r]^{\wt{\phi}} \ar[d]_{f} & N' \ar[d]^{g} \\ M \ar[r]^{\phi} & M'}$$
	If $\mathcal{E}$ and $\mathcal{F}$ are flat oriented line bundles over $M$ and $M'$, respectively, then there is an isomorphism
	$$(f,g)^* \colon \mathrm{Iso}_{\phi}(\mathcal{E},\mathcal{F}) \rightarrow \mathrm{Iso}_{\wt{\phi}}(f^*\mathcal{E},g^*\mathcal{F})$$
	uniquely defined by the requirement that the following diagram commutes for each $\Phi \in \mathrm{Iso}_{\phi}(\mathcal{E},\mathcal{F})$:
	$$\xymatrix{f^*\mathcal{E} \ar[r]^{(f,g)^*\Phi} \ar[d]_{F_{\mathrm{can}}} & g^*\mathcal{F} \ar[d]^{G_{\mathrm{can}}} \\ \mathcal{E} \ar[r]^{\Phi} & \mathcal{F}}$$
	Furthermore, if $\mathrm{Iso}_{\phi}(\mathcal{E},\mathcal{F})$ an $\R^+$-torsor, then so is $\mathrm{Iso}_{\wt{\phi}}(f^*\mathcal{E},g^*\mathcal{F})$, in which case $(f,g)^*$ is an isomorphism of $\R^+$-torsors.
\end{prop}

\begin{proof}
	For $x \in N$ and $\Phi \in \mathrm{Iso}_{\phi}(\mathcal{E},\mathcal{F})$, we simply define $(f,g)^*\Phi$ at $x$ by the composition
	$$(f^*\mathcal{E})_x \cong \mathcal{E}_{f(x)} \xrightarrow{\Phi_{f(x)}} \mathcal{F}_{\phi(f(x))} = \mathcal{F}_{(g(\wt{\phi}(x)))} \cong (g^*\mathcal{F})_{\wt{\phi}(x)}.$$
	The fact that $G_{\mathrm{can}} \circ (f,g)^*\Phi = \Phi \circ F_{\mathrm{can}}$ is obvious from construction. We note that $\mathrm{Iso}_{\phi}(\mathcal{E},\mathcal{F})$ is an $\R^+$-torsor if and only if $\phi^*\mathrm{hol}(\nabla_F) = \mathrm{hol}(\nabla_E)$. But this implies that
	$$\wt{\phi}^*\nabla_{g^*F} = \wt{\phi}^*g^*\nabla_F = f^*\nabla_E = \nabla_{f^*E},$$
	and so also $\mathrm{Iso}_{\wt{\phi}}(f^*\mathcal{E},g^*\mathcal{F})$ is an $\R^+$-torsor in this case. Finally, $(f,g)^*$ intertwines the $\R^+$ actions, and any map of $\R^+$-torsors intertwining the $\R^+$-action is an isomorphism.
\end{proof}

\begin{rem}
	\begin{itemize}
		\item In the case that $N = N'$, $M=M'$, and $f=g$, we simply write $f^*\Phi$ instead of $(f,g)^*\Phi$.
		\item In the case that $M = N = N'$, $f = \wt{\phi} = \mathrm{id}$, and $g = \phi$, we see that $\mathrm{Iso}_{\phi}(\scr{E},\scr{F}) \cong \mathrm{Iso}(\scr{E},\phi^*\scr{F})$.
	\end{itemize}
\end{rem}

The parallel transport idea is also relevant to the situation in which we have a family $\phi_t \colon M \rightarrow M$ which is the flow of a (possibly non-autonomous) vector field $X_t$. Namely, for each $p \in M$, we have a curve $\gamma(t) := \phi_t(p)$, and so parallel transport gives a linear map $F_{\gamma} \colon E_p \rightarrow E_{\phi_t(p)}$ for each $t$. We may view this as a bundle isomorphism $F_{\phi_t}$ of $E$ lying over $\phi_t$. In fact, we have that $F_{\phi_t}^*\nabla = \nabla$, so $F_{\phi_t} \in \mathrm{Iso}_{\phi_t}((E,\nabla),(E,\nabla))$.

\color{black}

\begin{dfn}
	Suppose $\scr{E} = (E,\nabla)$ is a flat line bundle over $M$. For a form $\omega \in \Omega^{\bullet}(M;E)$, its \textbf{twisted Lie derivative} is given by
 $$\scr{L}_X^{\scr{E}}\omega := \left.\frac{d}{dt}\right\vert_{t=0} F_{\phi_t}^*\omega \in \Omega^{\bullet}(M;E).$$
\end{dfn}

Cartan's magic formula extends to this setting, allowing for the following convenient computation:

\begin{thm}\label{thm:cartan}
	For $\omega\in \Omega(M;E)$ and $X\in \Gamma(TM)$ we have $$\mathcal{L}^{\scr{E}}_X\omega=d_\nabla(\iota_X\omega)+\iota_X d_\nabla\omega.$$.
\end{thm}

\begin{proof}
	Let us suppose we wish to compute $\scr{L}_X^{\scr{E}}\omega$ at a point $p \in M$. 
	Suppose $U$ is an open neighborhood around $p$. Then for some $\epsilon > 0$, we see that $\phi_t(p) \in U$ for all $|t| < \epsilon$. Furthermore, for this $\epsilon$, there is some open neighborhood $V$ around $p$ such that $\phi_t(V) \subseteq U$ for all $|t|<\epsilon$. (Notice, $V \subseteq U$ by considering $t=0$.) We see that for $|t|<\epsilon$, the value of $F_{\phi_t}^*\omega$ at $p$ only depends upon the family $\phi_t \colon V \rightarrow U$. In this sense the definition of $\scr{L}_X^\scr{E}\omega$ at $p$ is local: it only depends upon $\omega|_U$ and the family $\phi_t \colon V \rightarrow U$ for $|t| < \epsilon$. By choosing $U$ small enough so that the restriction of the holonomy class is trivial, Proposition \ref{prop:flat-line-bundle-facts} implies that $(E|_U,\nabla|_U)$ may be identified with the trivial flat bundle $(\underline{\R},d)$, and we may perform all of our computations with respect to this local trivialisation. We see that our definition for the twisted Lie derivative map is just the usual definition for the Lie derivative in this trivialisation, and that the formula we are trying to prove is just the usual Cartan formula, so the result holds directly from the untwisted setting.
\end{proof}

\subsection{The gauged viewpoint on flat oriented line bundles} \label{sec:gauge}

Suppose $\eta \in \Omega^1(M)$ is a closed differential form, and write $\rho := [\eta] \in H^1(M;\R)$ for its cohomology class. We shall refer to $\eta$ as a choice of gauge for $\rho$. Such a gauge naturally yields the \textbf{Lichnerowicz--de Rham complex}, given by
$$(\Omega^{\bullet}(M),d+\eta \wedge).$$
We note that if we choose a different representative $\eta+df$ for the same cohomology class, then multiplication by $e^{-f}$ naturally yields an isomorphism of cochain complexes
$$(\Omega^{\bullet}(M),d+\eta \wedge) \xrightarrow{\sim} (\Omega^{\bullet}(M),d+(\eta +df)\wedge).$$
If $\eta$ and $\eta'$ are cohomologous, then the collection of functions $f$ such that $\eta'-\eta = df$ is naturally an $\R$-torsor, and hence the complex of twisted differential forms is essentially an invariant of $(M,\rho)$, up to automorphisms given by multiplication by a positive constant. This is actually the same structure we saw for the twisted de Rham complex for flat oriented line bundles. This is more than a coincidence: these two theories describe the same invariant of $\rho$.

Indeed, suppose we start with a flat orientable line bundle $\pi \colon (E,\nabla) \rightarrow M$, and let $s$ be a nowhere zero section inducing a trivialization $E \cong \underline{\R}$, and hence $\Omega^{\bullet}(M;E) \cong \Omega^{\bullet}(M)$ as graded vector spaces. Then $d_\nabla$ is naturally identified with $d+\eta_s \wedge$, where $\eta_s\in \Omega^1(M)$ is a closed differential form with $e^{[\eta_s]} =  \mathrm{hol}(\nabla)$. Furthermore, $\eta_{e^fs} = \eta_s + df$. In this manner, we see that the twisted de Rham complex for the flat line bundle recovers the Lichnerowicz--de Rham complex. Conversely, starting from a closed $1$-form $\eta$, we may construct a flat oriented line bundle whose holonomy is $e^{[\eta]}$ by Proposition \ref{prop:flat-line-bundle-facts}.

For example, we recall that in the introduction, we defined LCS manifolds and LCS maps in terms of flat line bundles. We obtain an equivalent category for LCS geometry by using the following definitions:

\begin{dfn} \label{dfn:LCS_gauge}
	An \textbf{LCS manifold} is a triple $(M,\eta,\omega)$ such that:
	\begin{itemize}
		\item $\eta \in \Omega^1(M)$ is closed, i.e. $d\eta = 0$
		\item $\omega \in \Omega^2(M)$ is:
		\begin{itemize}
			\item \textbf{twisted-closed}, meaning $d\omega+ \eta \wedge \omega = 0$
			\item \textbf{non-degenerate}, meaning $\iota_{\bullet}\omega \colon TM \rightarrow T^*M$ is an isomorphism.
		\end{itemize}
	\end{itemize}
	An \textbf{LCS map} $\phi \colon (M_1,\eta_1,\omega_1) \rightarrow (M_2,\eta_2,\omega_2)$ is a map $\phi \colon M_1 \rightarrow M_2$ such that there exists a function $f \in C^{\infty}(M_1)$ with
	$$\phi^*\eta_2 = \eta_1+df\qquad \mathrm{and} \qquad \phi^*\omega_2 = e^f\omega_1.$$
	An LCS automorphism with underlying map given by the identity is called a \textbf{gauge equivalence}.
\end{dfn}

In what follows, we will mostly discuss our constructions and examples from the perspective of flat line bundles. However, for the purposes of concrete computations, when appropriate, we will recast these discussions in the gauged perspective.

\section{Twisted products of flat line bundles}
\label{sec:twisted-product-flat}

Recall that if $M_1$ and $M_2$ are smooth manifolds, then we have a natural product $M_1 \times M_2$ with comes with projections $\pi_j \colon M_1 \times M_2 \rightarrow M_j$. This allows us quite natural to mix the de Rham theory of $M_1$ and $M_2$ via the following two operations:
\begin{itemize}
	\item \textbf{Product:} We have a cochain map $$\Omega^{\bullet}(M_1) \otimes \Omega^{\bullet}(M_2) \rightarrow \Omega^{\bullet}(M_1 \times M_2)$$ given by
	$$(\alpha,\beta) \mapsto \pi_1^*\alpha \wedge \pi_2^*\beta.$$
	\item \textbf{Sum:} We have a cochain map $$\Omega^{\bullet}(M_1) \oplus \Omega^{\bullet}(M_2) \rightarrow \Omega^{\bullet}(M_1 \times M_2)$$ given by
	$$(\alpha,\beta) \mapsto \pi_1^*\alpha + \pi_2^*\beta.$$
\end{itemize}

Let us consider now the case of flat line bundles. Suppose that for $j=1,2$, we have that $(E_j,\nabla_j)$ is a line bundle over $M_j$ (where the projections $E_j \rightarrow M_j$ will be left implicit in the notation). The product naturally generalizes, yielding a cochain map
$$\Omega^{\bullet}(M_1;E_1) \otimes \Omega^{\bullet}(M_2;E_2) \rightarrow \Omega^{\bullet}(M_1 \times M_2; \pi_1^*E_1 \otimes \pi_2^*E_2)$$
given by the same formula
$$(\alpha,\beta) \mapsto \pi_1^*\alpha \wedge \pi_2^*\beta.$$
Unfortunately, however, the sum is not quite so easy. Indeed, $\pi_1^*\alpha$ is a section of $\pi_1^*E_1$, while $\pi_2^*\beta$ is a section of $\pi_2^*E_2$, and although we can multiply these just fine, if we try to add them, we see that $\pi_1^*E_1$ and $\pi_2^*E_2$ are not necessarily isomorphic, if for example $E_1$ is orientable and $E_2$ is not. Even if we restrict to the oriented case, we see that $\mathrm{hol}(\pi_1^*\nabla_1) = \pi_1^*\left(\mathrm{hol}(\nabla_1)\right)$ and $\mathrm{hol}(\pi_2^*\nabla_2) = \pi_2^*\left(\mathrm{hol}(\nabla_2)\right)$ may be different, and so trying to add $\pi_1^*\alpha$ and $\pi_2^*\beta$ is not a nice operation if we are looking for cochain maps.

The problem is that $M_1 \times M_2$ is not the right space on which to take sums. Instead, the underlying smooth space is what is constructed in Theorem \ref{thm:main_smooth}, if we take $\mathfrak{l}_1 = \ln(\mathrm{hol}(\nabla_1))$ and $\mathfrak{l}_2 = \ln(\mathrm{hol}(\nabla_2))$. First, in Section \ref{ssec:twisted_product}, we will define the twisted product smoothly in terms of these cohomology classes, proving Theorem \ref{thm:main_smooth}. Then, in Section \ref{ssec:tp_sum}, we will use Proposition \ref{prop:flat-line-bundle-facts} to work again with flat line bundles, allowing us to take sums of twisted differential forms upon an auxiliary but explicit choice. We end with Section \ref{ssec:tp_example}, in which we present an example of our construction.

\subsection{The twisted product} \label{ssec:twisted_product}

Let us recall a little covering space theory. If $N \trianglelefteq \pi_1(M,\ast)$ is a normal subgroup, then there is a covering space $q \colon \wt{M} \rightarrow M$ with deck group $\Gamma = \pi_1(M,\ast)/N$. If we are further given a cohomology class $\mathfrak{l} \in H^1(M;\R)$, which we think of as a group homomorphism $\rho_\mathfrak{l} \colon \pi_1(M,\ast) \rightarrow \R$, then so long as $0 \leq N \leq \mathrm{ker}(\rho_\mathfrak{l})$, we will have that $q^*\mathfrak{l} = 0 \in H^1(\wt{M};\R)$. We shall refer to such a connected cover with $q^*\mathfrak{l} = 0$ as a trivializing cover for the cohomology class $\mathfrak{l}$.

In this situation, any trivializing cover sits between the maximal one and minimal one. That is, it is covered by the universal cover corresponding to $N=0$, and it covers the minimal trivializing cover corresponding to $N = \ker(\rho_\mathfrak{l})$ (which have abelian deck group). We always have that $\rho_\mathfrak{l}$ factors through the projection $\pi_1(M,\ast) \rightarrow \pi_1(M,\ast)/N = \Gamma$ as a group homomorphism
$$\overline{\rho}_\mathfrak{l} \colon \Gamma \rightarrow \R.$$

Suppose now that we have manifolds $M_1$ and $M_2$ with cohomology classes $\mathfrak{l}_1 \in H^1(M;\R)$ and $\mathfrak{l}_2 \in H^1(M;\R)$. Pick trivializing covers $q_i \colon \wt{M}_i \rightarrow M_i$ with deck groups $\Gamma_i$, and let
$$\overline{\rho}_i \colon \Gamma_i \rightarrow \R$$
be the homomorphisms $\overline{\rho}_i := \overline{\rho}_{\mathfrak{l}_i}$. We obtain from these the fiber product
$$\Gamma_{12} := \{(g_1,g_2) \in \Gamma_1 \times \Gamma_2 \mid \overline{\rho}_1(g_1) = \overline{\rho}_2(g_2)\}$$
which comes with a map $\overline{\rho} \colon \Gamma_{12} \rightarrow \R$ given by $\overline{\rho}(g_1,g_2) := \overline{\rho}_1(g_1) = \overline{\rho}_2(g_2)$. In fact, $\Gamma_{12}$ is a normal subgroup of $\Gamma_1 \times \Gamma_2$. We define
$$M_1 \boxplus M_2 := \wt{M}_1 \times \wt{M}_2/\Gamma_{12}.$$
We note that there are still natural projections
$$r_i \colon M_1 \boxplus M_2 \rightarrow M_i.$$
A priori, this construction depends upon the choice of trivializing covers $q_1$ and $q_2$. However, it is actually independent of this data, as the following proposition shows:

\begin{prop}
	Let $\mathbf{M}_1$ and $\mathbf{M}_2$ be the universal covers of $M_1$ and $M_2$, respectively, with deck groups $\pi_1(M_i,\ast_i)$. Let $\rho_i:\pi_1(M_i,\ast_i)\rightarrow \mathbb{R}$ be given as $\rho_i := \rho_{\mathfrak{l}_i}$. We denote $$\mathbf{\Gamma}:=\{(g_1,g_2) \in \pi_1(M_1,\ast_1) \times \pi_1(M_2,\ast_2) |\mathbf{\rho}_1(g_1)=\mathbf{\rho}_2(g_2)\}.$$ Then we have $\mathbf{M}_1\times \mathbf{M}_2/\mathbf{\Gamma}\cong M_1\boxplus M_2$.
\end{prop}
\begin{proof}
	Let $g:\pi_1(M_1,\ast_1)\times \pi(M_2,\ast_2)\rightarrow  \Gamma_1\times \Gamma_2$ be the standard projection. This yields a covering map $\mathbf{M}_1\times \mathbf{M}_2\rightarrow \widetilde{M}_1\times \widetilde{M}_2$ such that $g$ intertwines the deck group action as covering maps over $M_1 \times M_2$. Since $g(\mathbf{\Gamma})=\Gamma_{12}$, we have natural identifications
	$$\mathbf{M_1} \times \mathbf{M_2}/\mathbf{\Gamma} \cong \wt{M}_1 \times \wt{M}_2/\Gamma_{12} = M_1 \boxplus M_2.$$
\end{proof}

\begin{cor} \label{cor:cohom_on_tp}
	The map
	$$r := (r_1,r_2) \colon M_1 \boxplus M_2 \rightarrow M_1 \times M_2$$
	is a covering map, and $r_1^*\mathfrak{l}_1 = r_2^*\mathfrak{l}_2$.
\end{cor}
\begin{proof}
	The fact that $r$ is a covering map is evident by construction. By the previous proposition, we have a commutative diagram
	$$\xymatrix{\pi_1(M_1 \boxplus M_2) \ar@/_1pc/[rdd]_-{(r_2)_*} \ar@/^1pc/[rrd]^{(r_1)_*} \ar[rd]^{\sim} & & \\ & \mathbf{\Gamma} \ar[r] \ar[d] & \pi_1(M_1,\ast_1) \ar[d]^-{\rho_1} \\ & \pi_1(M_2,\ast_2) \ar[r]^-{\rho_2} & \R}$$
	In particular, we have
	$$\rho_{r_1^*\mathfrak{l}_1} = \rho_1 \circ (r_1)_* = \rho_2 \circ (r_2)_* = \rho_{r_2^*\mathfrak{l}_2}.$$
	Thus $r_1^*\mathfrak{l}_1 = r_2^*\mathfrak{l}_2$.
\end{proof}

We may now prove Theorem \ref{thm:main_smooth}.

\begin{proof}[Proof of Theorem \ref{thm:main_smooth}]
We have constructed the covering map $r \colon M_1 \boxplus M_2 \rightarrow M_1 \times M_2$ as desired, and it is indeed canonical up to deck transformations. Item (1) is proved as Corollary \ref{cor:cohom_on_tp}, where the minimality of the covering space is precisely due to the fact that $\mathbf{\Gamma}$ is defined as a pullback.

Let us construct the embeddings $\Gamma(\phi)$ appearing in Item (2). Suppose that we have a smooth map $\phi \colon M_1 \rightarrow M_2$ with $\phi^*\mathfrak{l}_2 = \mathfrak{l}_1$. As above, we have
$$\pi_1(M_1 \boxplus M_2) \cong \mathbf{\Gamma} := \{(g_1,g_2) \in \pi_1(M_1) \times \pi_1(M_2) \mid \rho_1(g_1) = \rho_2(g_2)\}.$$
Suppose $g \in \pi_1(M_1)$, and consider the element
$$(\Gamma_0(\phi))_*(g) \in \pi_1(M_1 \times M_2) = \pi_1(M_1) \times \pi_1(M_2).$$
If $p_i \colon M_1 \times M_2 \rightarrow M_i$ are the usual projections, then $p_1 \circ \Gamma_0(\phi) = \mathrm{Id}$ and $p_2 \circ \Gamma_0(\phi) = \phi$, and hence
$$(\Gamma_0(\phi))_*(g) = (g,\phi_*(g)).$$
But since $\phi^*\mathfrak{l}_2 = \mathfrak{l}_1$, we have
$$\rho_1(g) = \rho_2(\phi_*(g)),$$
and so $(\Gamma_0(\phi))_*$ has image in $\mathbf{\Gamma}$. Hence, $\Gamma_0(\phi)$ lifts to a smooth map
$$\Gamma(\phi) \colon M_1 \rightarrow M_1 \boxplus M_2,$$
uniquely up to deck transformations of $r$, and which is an embedding since $\Gamma_0(\phi)$ is an embedding. This completes the proof of (2).

Finally, for Item (3), notice that we may take the same covering $q \colon \wt{M} \rightarrow M$ in our construction. Let us take $q$ to be the minimal such cover, so that the homomorphism $\overline{\rho}_{\mathfrak{\ell}} \colon \Gamma \rightarrow \R$ is injective. We see then that
$$M \boxplus M = \wt{M} \times \wt{M}/\Gamma,$$
where $\Gamma$ acts freely by the diagonal action $g \cdot (x,y) = (g \cdot x, g \cdot y)$. For a family $\phi_t$ of diffeomorphisms of $M$, note that by path-lifting, we obtain a canonical diffeomorphism $\wt{\phi}_t$ of $\wt{M}$, where $\wt{\phi}_t(x) = \phi_t(q(x))$. Furthermore, this is equivariant, in the sense that
$$\wt{\phi}_t(g \cdot x) = g \cdot \wt{\phi}_t(x).$$
We define $\wt{\psi}_t \in \mathrm{Diff}(\wt{M} \times \wt{M})$ by
$$\wt{\psi_t}(x,y) = (x,\wt{\phi}_t(x)).$$
This is equivariant with respect to the diagonal action of $\Gamma$ on $\wt{M} \times \wt{M}$, and hence descends to the desired diffeomorphisms $\psi_t \in \mathrm{Diff}(M \boxplus M)$. We clearly have $r_1 \circ \psi_t = r_1 = \mathrm{Id} \circ r_1$ and $r_2 \circ \psi_t = \phi_t \circ r_2$, which is to say that $\psi_t$ is a lift of $\mathrm{Id} \times \phi_t$ as desired.

Notice that we have the diagonal map $\wt{M} \xrightarrow{\Delta} \wt{M} \times \wt{M}$ which is equivariant with respect to the natural action of $\Gamma$, and hence descends to a map $M \rightarrow M \boxplus M$. We see that we may take this as $\Gamma(\mathrm{Id})$. On the other hand, we have that
$$r_1 \circ (\psi_t \circ \Gamma(\mathrm{Id})) = (r_1 \circ \psi_t) \circ \Gamma(\mathrm{Id}) = r_1 \circ \Gamma(\mathrm{Id}) = \mathrm{Id}$$
$$r_2 \circ (\psi_t \circ \Gamma(\mathrm{Id})) = (r_2 \circ \psi_t) \circ \Gamma(\mathrm{Id}) = (\phi_t \circ r_2) \circ \Gamma(\mathrm{Id}) = \phi_t \circ (r_2 \circ \Gamma(\mathrm{Id})) = \phi_t \circ \mathrm{Id} = \phi_t.$$
Hence, $r \circ (\psi_t \circ \Gamma(\mathrm{Id})) = \Gamma_0(\phi_t)$, which is to say that $\psi_t \circ \Gamma(\mathrm{Id})$ is a canonical choice of $\Gamma(\phi_t)$.

Finally, suppose we have an intersection point of $\Gamma(\mathrm{Id})$ and $\Gamma(\phi_1)$. Since these are embeddings, such intersections are in bijection with pairs $(x,y) \in (M,M)$ such that
$$\Gamma(\mathrm{Id})(x) = \Gamma(\phi_1)(y).$$
Applying $r_1$ to both sides, we see that $x=y$ Hence, such intersection points are in bijection with those points $x \in M$ with
$$\Gamma(\mathrm{Id})(x) = \Gamma(\phi_1)(x),$$
i.e. for which
$$\Gamma(\mathrm{Id})(x) = \psi_1(\Gamma(\mathrm{Id})(x)).$$
Applying $r_2$, we find that $x = \phi_1(x)$, so $x \in \mathrm{Fix}(\phi_1)$. Let $\wt{x} \in q^{-1}(x) \subset \wt{M}$ be any fixed lift, so that $(\wt{x},\wt{x}) \in \wt{M} \times \wt{M}$ is a lift of $\Gamma(\mathrm{Id}(x))$. Then the above equation is satisfied if there is some $g \in \Gamma$ such that
$$\wt{\psi}_1(\wt{x},\wt{x}) = g \cdot (\wt{x},\wt{x}).$$
We see that since the action of $\Gamma$ is free, this is equivalent to requiring $\wt{\phi}_t(\wt{x}) = \wt{x}$. But we have $\wt{\phi}_t(\wt{x}) = g \cdot \wt{x}$ where $g$ is the image of the path $[\phi_t(x)]$ under the quotient map $\pi_1(M) \rightarrow \Gamma$. We have $g$ is the identity if and only if $[\phi_t(x)]$ is in the kernel of this map, which is to say that
$$\langle \mathfrak{l}, [\phi_t(x)] \rangle = 0.$$
Hence, intersection points of $\Gamma(\mathrm{Id})$ and $\Gamma(\phi_1)$ are in bijection with those fixed points $x \in M$ with $\langle \mathfrak{l}, [\phi_t(x)] \rangle = 0$. The correspondence between transverse intersections and non-degenerate fixed points is the same as for the usual product since $r$ is a local diffeomorphism.
\end{proof}

\subsection{Adding twisted differential forms} \label{ssec:tp_sum}

Let us now get back to our desire to add twisted differential forms. Suppose that $\pi_i \colon (E_i,\nabla_i) \rightarrow M_i$ is a flat line bundle, and set $\mathfrak{l}_i := \ln(\mathrm{hol}(\nabla_i)) \in H^1(M_i;\R)$. Our sum construction produced a covering map
$$r = (r_1,r_2) \colon M_1 \boxplus M_2 \rightarrow M_1 \times M_2$$
such that $r_1^*\mathfrak{l}_1 = r_2^*\mathfrak{l}_2$. In particular, we have that
$$\mathrm{hol}(r_1^*\nabla_1) = \mathrm{hol}(r_2^*\nabla_2) \in H^1(M_1 \boxplus M_2; \R^+).$$
By Proposition \ref{prop:flat-line-bundle-facts}, this means there is a corresponding $\R^+$-torsor of isomorphisms of oriented flat line bundles $r_2^*(E_2,\nabla_2) \xrightarrow{\sim} r_1^*(E_1,\nabla_1)$. We are finally able to define our desired sum on twisted differential forms.

\begin{dfn}
	Let $\Xi$ be the $\R^+$-torsor of isomorphisms of oriented flat line bundles
	$$\{\Phi:(r_2^*E_2,r_2^*\nabla_2)\xrightarrow{\sim} (r_1^*E_1,r_1^*\nabla_1)\}.$$
	Then for each $\Phi \in \Xi$, we define the chain map
	$$\boxplus_{\Phi}:\Omega^{\bullet}(M_1;E_1)\oplus \Omega^{\bullet}(M_2:E_2)\rightarrow \Omega^{\bullet}(M_1\boxtimes M_2;r_1^*E_1)$$
	given by
	$$\alpha \boxplus_{\Phi} \beta := r_1^*\alpha+\Phi(r_2^*\beta).$$
\end{dfn}

\begin{rem}
	The construction is symmetric in the sense that we may alternatively take the codomain as $\Omega^{\bullet}(M_1 \boxplus M_2; r_2^*E_2)$, using the formula $\Phi^{-1}(r_1^*\alpha) + r_2^*\beta$.
\end{rem}

Let us consider the gauged perspective for this twisted product of forms. Suppose we have fixed oriented trivialisations
$$s_j \colon E_j \xrightarrow{\sim} \underline{\R}$$
for $j=1,2$ of oriented line bundles over $M_1$ and $M_2$. Notice that these give also trivialisations
$$r_j^*s_j \colon r_j^*E_j \xrightarrow{\sim} \underline{\R}$$
over $M_1 \boxplus M_2$. In particular, if the trivialisation $s_j$ intertwines $d_{\nabla_j}$ with $d + \eta_j$ for a closed $1$-form $\eta_j \in \Omega^1(M_j)$, then $r_j^*s_j$ intertwines $d_{r_j^*\nabla_j}$ with $d+r_j^*\eta_j \wedge$.

Suppose now that $\Xi$ is the $\R^+$-torsor of bundle isomorphisms
$$\Xi := \mathrm{Iso}((r_2^*E_2,r_2^*\nabla_2),(r_1^*E_1,r_1^*\nabla_1)).$$
For each $\Phi \in \Xi$, we have a composition of morphisms of oriented line bundles over $M_1 \boxplus M_2$
$$\underline{\R} \xrightarrow{(r_2^*s_2)^{-1}} r_2^*E_2 \xrightarrow{\Phi} r_1^*E_1 \xrightarrow{r_1^*s_1} \underline{\R}.$$
Including the connections, this is an element
$$\Phi' \in \mathrm{Iso}((\underline{\R},d+r_2^*\eta_2 \wedge), (\underline{\R},d+r_1^*\eta_1 \wedge)).$$
As an isomorphism of trivial oriented line bundles, this is just multiplication by some positive function, and so we may equivalently think of $\Phi'$ as a positive function on $C^{\infty}(M_1 \boxplus M_2)$. Explicitly, we have that $r_1^*\eta_1$ and $r_2^*\eta_2$ are cohomologous since they represent the logarithms of the holonomies of $r_1^*\nabla_1$ and $r_2^*\nabla_2$. As such, there is an $\R$-torsor of functions $f \in C^{\infty}(M_1 \boxplus M_2)$ with $r_2^*\eta_2 = r_1^*\eta_1+df$. The functions $e^f$ naturally define an $\R^+$-torsor $\Xi'$, and we obtain an isomorphism of $\R^+$-torsors $\chi \colon \Xi \xrightarrow{\sim} \Xi'$ given by $\chi(\Phi) = \Phi'$.

For each $\Phi' \in \Xi'$ with $\Phi' = \chi(\Phi)$ for some $\Phi \in \Xi$, we therefore obtain a map of cochain complexes $\boxplus_{\Phi'}$ given by the dotted arrow in the following diagram:
$$\xymatrix{(\Omega^{\bullet}(M_1;E_1),\nabla_1)\oplus (\Omega^{\bullet}(M_2;E_2),\nabla_2)\ar^-{\boxplus_{\Phi}}[r] \ar[d]_{s_1 \oplus s_2} & (\Omega^{\bullet}(M_1\boxplus M_2;r_1^*E_1),r_1^*\nabla_1) \ar[d]^{r_1^*s_1} \\ (\Omega^{\bullet}(M_1),d+\eta_1 \wedge) \oplus (\Omega^{\bullet}(M_2), d+\eta_2 \wedge) \ar@{-->}^-{\boxplus_{\Phi'}}[r] & (\Omega^{\bullet}(M_1 \boxplus M_2),d+r_1^*\eta_1 \wedge )}.$$
Thinking of $\Phi'$ as a positive function in $C^{\infty}(M_1 \boxplus M_2)$, the operation $\boxplus_{\Phi'}$ is just
$$\alpha \boxplus_{\Phi'} \beta = r_1^*\alpha + \Phi' \cdot r_2^*\beta.$$
We see that since in the diagram, the vertical arrows are isomorphisms, we have that $\boxplus_{\Phi}$ and $\boxplus_{\Phi'}$ determine each other. Of course, the $\R^+$-torsor $\Xi'$ depends on $\eta_1$ and $\eta_2$, and the identification $\chi \colon \Xi \xrightarrow{\sim} \Xi'$ depends upon the choice of trivialisations $s_1$ and $s_2$ which intertwine $d_{\nabla_j}$ with $d+\eta_j \wedge$.

\subsection{An example} \label{ssec:tp_example}

To finish this section, we give an example that will be relevant in the future.

\begin{ex}\label{ex-stabilisation}
	On $S^1 = \R/\Z$, any flat oriented line bundle is isomorphic to $(\underline{\R},\nabla_{\alpha})$ for some $\alpha \in \R$, where
	$$\langle \ln(\mathrm{hol}(\nabla_{\alpha})), \gamma \rangle = \alpha$$
	for $\gamma$ a loop winding around $S^1$ once. Explicitly, since the bundle is trivial, we have an identification $\Omega^{\bullet}(S^1;\underline{\R}) = \Omega^{\bullet}(S^1)$, and $\nabla_{\alpha}$ is the map $\Omega^{0}(S^1) \rightarrow \Omega^1(S^1)$ with
	$$\nabla_{\alpha}(f) = df + \alpha f d\theta.$$
	
	Suppose also that on a manifold $M$, we have a flat oriented line bundle $(E,\nabla_E)$ with $\ln(\mathrm{hol}(\nabla_E)) = \mathfrak{l}$ such that the period subgroup
	$$\mathrm{im}(\mathfrak{l} \colon \pi_1(M) \rightarrow \R)$$
	is given by $\Z\alpha$ for the same $\alpha$. (The example of $S^1$ with $(\underline{\R},\nabla_{\alpha})$ is an example of this.) Let us further suppose that $\alpha \neq 0$. In such a setting, let us compute $M \boxplus S^1$.
	
	Let $p:\widetilde{M}\rightarrow M$ be the minimal cover so that $p^*\mathfrak{l}=0$. The deck transformation group is $\mathbb{Z}$ and the morphism $\mathfrak{l} \colon \pi_1(M) \rightarrow \R$ induces the map $\overline{\mathfrak{l}} \colon \Z \rightarrow \R$ given by $k \mapsto k\alpha$. Let $f:\widetilde{M}\rightarrow \mathbb{R}$ be any map so that $f(k\cdot q)=f(q)+k$. Explicitly, $f$ may be taken as follows: if $\eta \in \Omega^1(M)$ is a closed $1$-form representing the cohomology class $\mathfrak{l} \in H^1(M;\R)$, then $p^*\eta$ is exact, and we may take our function $f$ as any of the solutions (determined up to a constant) to the equation $\alpha df = p^*\eta$. For $S^1$ with $(\underline{\R},\nabla_{\alpha})$, the minimal cover is the universal one $p \colon \R \rightarrow S^1$.
	
	In order to form the product $M \boxplus S^1$, we therefore need to take the quotient of $\wt{M} \times \R$ by the corresponding group $\Gamma_{12}$, which is just the diagonal $\Z \leq \Z \times \Z$. In other words, we have an action of $\Z$ on $\wt{M} \times \R$ given by
	$$k \cdot (q,t) = (k \cdot q, t+k)$$
	and we have
	$$M \boxplus S^1 = (\wt{M} \times \R)/\Z.$$
	The function $f$ gives us a diffeomorphism of $\wt{M} \times \R$ given by $(q,t) \mapsto (q,t+f(q))$, which intertwines the natural action of $\Gamma_{12}=\Z$ with the action given by
	$$k \cdot (q,t) = (k \cdot q, t).$$
	Therefore, we obtain an identification $M \boxplus S^1 \cong M \times \R$. The corresponding projection $M \boxplus S^1 \rightarrow M \times S^1$ is identified with the projection
	$$M \times \R \rightarrow M \times S^1$$
	given by $(q,t) \mapsto (q,t+f(\wt{q}))$ (for any $\wt{q}$ with $p(\wt{q})=q$). In other words, we may think of $r = (r_1,r_2)$ with $r_1 \colon M \times \R \rightarrow M$ the usual projection and with $r_2 \colon M \times \R \rightarrow S^1$ given by $r_2(q,t) = t+f(\wt{q})$.


This is particularly nice from the gauged perspective. Under this identification $M \boxplus S^1 \cong M \times \R$, we find $r_1^*\eta = r_2^*(\alpha d\theta)-\alpha dt$ (where $t$ is the $\R$-coordinate), and correspondingly, each operation
	$$(\Omega^{\bullet}(M),d+\eta\wedge ) \oplus (\Omega^{\bullet}(S^1), d+\alpha d\theta \wedge) \xrightarrow{\boxplus_{\Phi}} (\Omega^{\bullet}(M \boxplus S^1), d+r_1^*\eta \wedge)$$
	is given by
	$$\mu \boxplus_{\Phi} \nu := r_1^*\mu + e^{c+\alpha t} \cdot r_2^*\nu$$
	for some $c \in \R$ (determining and determined by $\Phi \in \Xi$). Finally, we may think of this as a form on $\Omega^{\bullet}(M \times \R)$ via our identification before, so that $r_1^*\mu$ and $r_2^*\nu$ in the above formula may be explicitly computed. Since $r_1$ is just the usual projection, $r_1^*\mu$ is clear. On the other hand, we have for example
	$$r_2^*(gd\theta) = g(t+f(\widetilde{q})) \cdot (dt + \eta/\alpha).$$

\end{ex}
\color{black}

\section{LCS geometry}
\label{sec:LCS}

Recall from the introduction the definition of LCS manifolds and LCS maps. We shall often write $(M,E,\nabla,\omega)$ for the data of an LCS manifold, where the projection $(E,\nabla) \rightarrow M$ will be left implicit. Notice that an LCS map is an invertible if and only if the underlying map $\phi \colon M_1 \rightarrow M_2$ is a global diffemorphism; we shall refer to such invertible LCS maps as \textbf{LCS symplectomorphisms}.

We recall Remark \ref{rem:base_only}: although an LCS map consists of a bundle map, written as a pair $(\phi,\Phi)$, the underlying map $\phi$ at the level of manifolds determines $\Phi$. Indeed, since $\Phi$ is a morphism of flat oriented line bundles, it is determined up to a constant by Proposition \ref{prop:flat-line-bundle-facts}, and the condition $\Phi^*\omega' = \omega$ therefore determines $\Phi$ uniquely.

It is natural, therefore, to ask which maps $\phi$ of smooth manifolds are the base of LCS maps. As a specific example, fix an LCS manifold $\scr{M} = (M,E,\nabla,\omega)$, and suppose that $X_t$ is a (possibly non-autonomous) family of vector fields on $M$ with flow given by $\phi_t$. We recall from Section \ref{sec:flat-line-bundles} that parallel transport may be considered as a family $F_{\phi_t} \in \mathrm{Iso}_{\phi_t}((E,\nabla),(E,\nabla))$.

Furthermore, by Proposition \ref{prop:flat-line-bundle-facts}, any family of automorphisms of $(E,\nabla)$ covering $\phi_t$ must be of the form $e^{c_t}F_{\phi_t}$  for some family of constants $c_t$. Therefore, $\phi_t$ is a family of underlying LCS automorphisms of $\scr{M}$ if and only if there is a family of constants $c_t$ such that
$$(\phi_t,e^{c_t}F_{\phi_t}) \colon (M,E,\nabla,\omega) \rightarrow (M,E,\nabla,\omega)$$
is an LCS map. Setting
$$\omega_t := F_{\phi_t}^*\omega,$$
we find therefore that $\phi_t$ is a family of LCS maps if and only if
$$\omega_t = e^{c_t}\omega$$
for a family of constants $c_t$. Taking a derivative in $t$, this is the case if and only if
$$\left(\dot{c}_t\right) \cdot \omega_t = \scr{L}_{X_t}^{\scr{E}}\omega_t.$$
Applying the Cartan formula of Theorem \ref{thm:cartan}, and noting that $d_\nabla \omega_t = 0$, we find
$$\dot{c}_t \cdot \omega_t = d_\nabla(\iota_{X_t}\omega).$$
We obtain the following definition and proposition which summarizes this discussion.

\begin{dfn} \label{dfn:LCS_vf}
	A vector field $X \in \Gamma(TM)$ is an \textbf{LCS vector field} for an LCS structure $(M,E,\nabla,\omega)$ if
	$$d_\nabla(\iota_X\omega) = \mu_X \omega$$
	for a constant $\mu_X \in \R$ which we call the \textbf{expansion constant of $X$}.
\end{dfn}

\begin{prop}\cite[Lemma 1]{allais2022dynamics} \label{prop:LCS_vector_fields}
	Suppose $\scr{M} = (M,E,\nabla,\omega)$ is an LCS manifold, and let $\phi_t$ be the flow of a family of vector fields $X_t$ on $M$. Then $\phi_t$ extends (uniquely) to an LCS automorphisms $(\phi_t,\Phi_t)$ of $\scr{M}$ if and only if each $X_t$ is an LCS vector field. In such a case, the bundle map $\Phi_t$ is given by
	$$\Phi_t = e^{\int_0^{t}\mu_{X_{\tau}}d\tau}F_{\phi_t}.$$
\end{prop}

A specific example are the LCS Hamiltonian vector fields briefly discussed in the introduction.

\begin{dfn}\label{dfn:LCS-Hamiltonian}
	Suppose $H \in \Gamma(E)$ is a section of $E$, called an \textbf{LCS-Hamiltonian function}. Then it has an associated \textbf{LCS-Hamiltonian vector field} $X_H$ on $M$ defined by
	$$\iota_{X_H}\omega = -\nabla H.$$
	The flow of a family of LCS-Hamiltonian vector fields is called an \textbf{LCS-Hamiltonian isotopy}.
\end{dfn}

\begin{prop}\label{expans-ham}
	LCS Hamiltonian vector fields $X_H$ are LCS vector fields with expansion coefficients $\mu_{X_H} = 0$. In particular, if $\phi_t$ is an LCS-Hamiltonian isotopy, then $(\phi_t,\scr{F}_{\phi_t})$ is an LCS automorphism.
\end{prop}
\begin{proof}
	We have $d_\nabla(\iota_{X_H}\omega) = -d_\nabla^2\omega = 0.$ The final statement follows from Proposition \ref{prop:LCS_vector_fields}.
\end{proof}

One particular type of LCS manifold plays particularly nicely with our constructions, and also naturally gives rise to LCS vector fields.

\begin{dfn} \label{dfn:ELCS}
	An \textbf{exact LCS (ELCS) manifold} $(M,E,\nabla,\lambda)$ consists of the data of a smooth manifold $M$, a flat oriented line bundle $\pi \colon (E,\nabla) \rightarrow M$, and an $E$-valued differential $1$-form $\lambda \in \Omega^1(M;E)$ such that the quadruple $(M,E,\nabla,d_\nabla \lambda)$ is LCS. We refer to $\lambda$ as a \textbf{primitive} of the LCS structure determined by $d_\nabla \lambda$. An \textbf{ELCS map} $$(\phi,\Phi) \colon (M_1,E_1,\nabla_1,\lambda_1) \rightarrow (M_2,E_2,\nabla_2,\lambda_2)$$ is a morphism of flat oriented line bundles such that $$\Phi^*\lambda_2 = \lambda_1 + \nabla_1 s$$ for some $s \in \Gamma(E_1)$.
\end{dfn}

\begin{rem}
Observe that since $d_\nabla^2=0$ we have that an ELCS map is automatically an LCS map for the induced LCS structures.
\end{rem}

\begin{dfn}
The \textbf{Liouville vector field} on an ELCS manifold is the vector field $V_{\lambda} \in \Gamma(TM)$ uniquely defined by the equation
$$\iota_{V_{\lambda}}d_\nabla \lambda = \lambda.$$
\end{dfn}

Observe that $V_{\lambda}$ is an LCS vector field with expansion constant $\mu_{V_{\lambda}} = 1$. Hence, the flow of $V_{\lambda}$ generates LCS-symplectomorphisms.

\begin{prop} \label{prop:covering_map}
  Suppose $\pi \colon \wt{X} \rightarrow X$ is a connected covering map, and suppose that we have fixed an LCS structure $(E,\nabla,\omega)$ on $\wt{X}$. Suppose further that the deck group of $\pi$ acts by LCS maps. Then there is an LCS structure on $X$ such that $\pi$ is an LCS map. Furthermore, the LCS structure is unique up to unique isomorphism, in the sense that if there are two such LCS structures $(E_1,\nabla_1,\omega_1)$ and $(E_2,\nabla_2,\omega_2)$ (for which $\pi$ is an LCS map), then the identity on $X$ extends uniquely to an LCS symplectomorphism $(X,E_1,\nabla_1,\omega_1) \rightarrow (X,E_2,\nabla_2,\omega_2)$.
  
  Conversely, if we have a connected covering map $\pi \colon \wt{X} \rightarrow X$ and $X$ has an LCS structure, then there is similarly an essentially unique LCS structure on $\wt{X}$ such that $\pi$ is an LCS map.
\end{prop}
\begin{proof}
	For the forward direction, notice that since the deck group action on $\wt{X}$ extends to a flat oriented line bundle isomorphism of $(E,\nabla)$, we have that one may construct a flat oriented line bundle $(\overline{E},\overline{\nabla})$ on $X$ together with an isomorphism of flat oriented line bundles $\Psi \in \mathrm{Iso}(\pi^*(\overline{E},\overline{\nabla}),(E,\nabla))$.
	Suppose $\Gamma$ is the deck group of $\pi$.  Let $\phi_g\in\Gamma$, and let $(\phi_g,\Phi_g)$ denote the corresponding LCS map of $\wt{X}$. Let
	$$\Psi_g := \Psi^{-1}\Phi_g\Psi \in \mathrm{Iso}_{\phi_g}(\pi^*(\overline{E},\overline{\nabla}),\pi^*(\overline{E},\overline{\nabla})).$$
Then
	$$\Psi_g^*(\Psi^*\omega) = \Psi^*\omega.$$
	But also notice by Proposition \ref{prop:restrict_torsor} \color{black} that there is another flat oriented line bundle automorphism
	$$\mathcal{I}_g := \pi^*\mathrm{Id} \in \mathrm{Iso}_{\phi_g}(\pi^*(\overline{E},\overline{\nabla}),\pi^*(\overline{E},\overline{\nabla})).$$
	We have by Proposition \ref{prop:flat-line-bundle-facts} that $\Psi_g = e^{c_g}\mathcal{I}_g$. Furthermore, it is easy to check that the constants $c_g$ are given as $c_g = \rho(g)$ for a group homomorphism $\Gamma \xrightarrow{\rho} \R$.
	
	The composition $\pi_1(X) \rightarrow \Gamma \xrightarrow{\rho} \R$ yields a first cohomology class $\theta \in H^1(X;\R)$ such that $\pi^*\theta = df$ for some $f \in C^{\infty}(\wt{X})$, satisfying $f\circ \phi_g = f+c_g$. We therefore have that
	$$\scr{I}_g^*\left(e^{-f}\Psi^*\omega\right) = e^{-f}\Psi^*\omega,$$
	and hence $e^{-f}\Psi^*\omega = \pi^*\overline{\omega}$ for some $\overline{\omega} \in \Omega^2(X;\overline{E})$. Furthermore, if we take
	$$\overline{\nabla}_\theta :=\nabla + \theta,$$
	i.e. such that $\overline{\nabla}_{\theta}s = \overline{\nabla}s + \theta \otimes s$, then we obtain a chain of LCS maps
	$$\xymatrix{(E,\nabla,\omega) \ar[r]^-{\Psi^{-1}} \ar[d] & (\pi^*\overline{E},\pi^*\overline{\nabla},\Psi^*\omega) \ar[r]^-{e^{-f}} \ar[d] & (\pi^*\overline{E},\pi^*\overline{\nabla}_{\theta},e^{-f}\Psi^*\omega) \ar[r]^-{\Pi_{\mathrm{can}}} \ar[d] & (\overline{E},\overline{\nabla}_{\theta}, \overline{\omega}) \ar[d] \\ \wt{X} \ar@{=}[r] & \wt{X} \ar@{=}[r] & \wt{X} \ar[r]^{\pi} & X}$$
	We see that the base map is just $\pi$, and this hence completes our construction.
	
	Suppose now we have two such quotients $(E_1,\nabla_1,\omega_1)$ and $(E_2,\nabla_2,\omega_2)$ on $X$. For each open set $U \subset X$ evenly covered by $\pi$, we may choose a lift $V \subset \wt{X}$ so that $\pi|_V \colon V \rightarrow U$ is a diffeomorphism. Then we construct an LCS isomorphism locally over $U$ by restricting the LCS map $\pi$ to $V$. That is, we have
	$$\xymatrix{(E_1,\nabla_1,\omega_1)|_U \ar[d] & (E,\nabla,\omega)|_V \ar[r]^-{\sim} \ar[l]_-{\sim} \ar[d] & (E_2,\nabla_2,\omega_2)|_U \ar[d] \\ U & V \ar[l]_{\pi|_V} \ar[r]^{\pi|_V} & U}$$
	where may view the top row as defining an LCS map over the identity in $U$ from $(E_1,\nabla_1,\omega_1)$ to $(E_2,\nabla_2,\omega_2)$. We note that this construction was independent upon the choice of lift $V$, and hence extends to an LCS symplectomorphism over all of $X$. Uniqueness has already been discussed: any LCS symplectomorphism is determined by its base map, in this case the identity on $X$.
	
	As for the converse, we see that if $(E,\nabla,\omega)$ is an LCS structure on $X$, then $(\pi^*E,\pi^*\nabla,\pi^*\omega)$ is an LCS structure on $\wt{X}$ with $\pi \colon \wt{X} \rightarrow X$ inducing an LCS map between them. Uniqueness up to unique isomorphism is similar to the above, but where instead we use the diagram of LCS maps
	$$\xymatrix{(E_1,\nabla_1,\omega_1) \ar[d] \ar[r]^-{\sim} & (E,\nabla,\omega) \ar[d] & (E_2,\nabla_2,\omega_2) \ar[d] \ar[l]_-{\sim} \\ \wt{X} \ar[r]^{\pi} & X & \wt{X} \ar[l]_{\pi}}.$$
	
\end{proof}

\color{black}

There is a further refinement to the case of ELCS manifolds. We will say that a primitive $\lambda$ of an LCS manifold $(M,E,\nabla,\omega)$ is \textbf{equivariant} under an LCS automorphism $(\phi,\Phi)$ if $\Phi^*\lambda = \lambda$. (We say \emph{equivariant} instead of \emph{invariant} because implicitly there are some constant scaling factors encoded in $\Phi$.)
\begin{prop}\label{prop:covering_map_exact}
Let $\pi \colon \wt{X} \rightarrow X$ be a connected covering map. Fix an ELCS structure $(E,\nabla,\lambda)$ on $\wt{X}$. If the deck group acts by LCS maps for which $\lambda$ is equivariant, then there is an induced ELCS structure induced on $X$, unique up to unique ELCS isomorphism, for which $\pi$ is an ELCS map. Conversely, an ELCS on $X$ lifts to a unique ELCS structure up to unique isomorphism on $\wt{X}$ so that $\pi$ is an ELCS map, and the primitive is equivariant.
	
\end{prop}
\begin{proof}
	The proof is essentially the same as in Proposition \ref{prop:covering_map}, but where we replace $\omega$ everywhere with $\lambda$.
\end{proof}

We end this section with two examples of LCS manifolds that will be relevant for our discussion of products.

\begin{ex}\label{excotan}
	Suppose $Q$ is a manifold and $(E,\nabla)$ is a flat oriented line bundle over it. Then we may form an ELCS manifold $T^*_{(E,\nabla)}Q$, called the \textbf{twisted cotangent bundle}, as follows. The underlying manifold is the total space of the bundle $$\pi \colon T^*Q \otimes E \rightarrow Q,$$ and it comes with the flat line bundle $(\pi^*E,\pi^*\nabla)$. Just like the usual cotangent bundle, there is a canonical $1$-form $\lambda^E \in \Omega^1(T^*Q \otimes E; \pi^*E)$, such that its value at $p \otimes s \in T^*_qQ \otimes E_q$ is given by $\pi^*p \otimes \pi^*s$. That is, if $X \in T_{p \otimes s}(T^*Q \otimes E))$, then
	$$\lambda^E(X) = p(d\pi(X)) \cdot \pi^*s \in (\pi^*E)_{p \otimes s}$$
	where $\pi^*s$ refers to the image of $s$ under the identification $E_q \cong (\pi^*E)_{(q,p \otimes s)}$. The form $\lambda^E$ may be called the \textbf{tautological $1$-form}, since if we have any $E$-valued $1$-form $\alpha \in \Omega^1(Q;E)$, it corresponds to a section $\wt{\alpha} \in \Gamma(Q;T^*Q \otimes E)$, and it satisfies
	$$\wt{\alpha}^*\lambda^E = \alpha.$$
	The differential $\omega := -d_{\pi^*\nabla} \lambda^E$ is non-degenerate. Indeed, the construction is natural with respect to morphisms of flat line bundles, and so since any flat line bundle admits a local trivialisation, our construction matches the usual symplectic form on cotangent bundles. In this manner, we have defined an ELCS structure. The Liouville vector field is the radial vector field tangent to each fiber $T_q^*Q \otimes E_q$. Hamiltonian dynamics on such twisted cotangent bundles was studied in \cite{zbMATH07103007}.
	
	The minimal trivializing cover of the twisted cotangent bundle $T^*Q \otimes E$ is $T^*\wt{Q} \otimes \pi^*E$, where $\pi \colon \wt{Q} \rightarrow Q$ is the minimal trivializing cover for the flat line bundle $(E,\nabla)$, say with deck group $\Gamma$. In particular, an identification of $(\pi^*E,d_{\pi^*\nabla}) \cong (\underline{\R},d)$ identifies this cover with $T^*\wt{Q}$, such that the LCS structure of Proposition \ref{prop:covering_map} is the usual canonical symplectic form up to global scale. With this identification, the deck group $\Gamma$ acts on $(q,p) \in T^*\wt{Q}$ (with $q \in \wt{Q}$ and $p \in T^*_q\wt{Q}$) by
	$$g \cdot (q,p) = (g \cdot q, e^{\rho(g)}p)$$
	where $\rho \colon \Gamma \rightarrow \R$ is the morphism which composes with the projection $\pi_1(Q) \rightarrow \Gamma$ to give the morphism $\ln(\mathrm{hol}(\nabla)) \colon \pi_1(Q) \rightarrow \R$. This gives an alternative definition of $T^*_{(E,\nabla)}Q$ as the quotient $T^*\wt{Q}/\Gamma$ as in Proposition \ref{prop:covering_map}.
	
	Finally, we have a gauge theoretic interpretation. Suppose we fix a trivialisation of the oriented line bundle $E$, yielding an identification $\Phi \colon (E,d_\nabla) \xrightarrow{\sim} (\underline{\R},d+\eta \wedge)$ for a closed $1$-form $\eta$. Then we have an identification
	$$(T^*Q \otimes E,\pi^*E, \pi^*\nabla,\lambda^E) \cong (T^*Q,\underline{\R},d+\pi^*\eta, \lambda^{\mathrm{can}})$$
	where $\lambda^{\mathrm{can}}$ is the standard tautological $1$-form on $T^*Q$.	In particular, we may write our ELCS manifold in the gauge-theoretic way as
	$$(T^*Q,\pi^*\eta, \omega = d_{\pi^*\eta}\lambda^{\mathrm{can}})$$
	where we note that $\omega = \omega^{\mathrm{can}} + \pi^*\eta \wedge \lambda^{\mathrm{can}}$ is not just the usual canonical symplectic structure. (In fact, $d\omega \neq 0$ in general.)

    In the standard symplectic setting, diffeomorphisms of $Q$ lift to exact symplectomorphisms of $T^*Q$. We describe now an analogous version for the ELCS structure on $T^*_{(E,\nabla)}Q$. Suppose that $\phi$ is a diffeomorphism of $Q$ which is \textbf{holonomy preserving}, meaning that
        $$\mathrm{hol}(\phi^*\nabla) = \phi^*(\mathrm{hol}(\nabla)) = \mathrm{hol}(\nabla).$$
        By Proposition \ref{prop:flat-line-bundle-facts}, we may extend $\phi$ to an oriented flat line bundle isomorphism $\Phi \colon (E,\nabla) \rightarrow (E,\nabla)$ lying over $\phi$, where $\Phi$ is determined up to multiplication by a positive constant.
        
        Recall that for each $q \in M$, we have an isomorphism
        $$\Phi_q^* \colon T^*_{\phi(q)}Q \otimes E_{\phi(q)} \rightarrow T^*_{q}Q \otimes E_q.$$
        Since $\phi$ is a diffemorphism, we may glue these together to obtain a diffeomorphism
        $$\wt{\phi} = (\Phi^*)^{-1} \colon T^*_{(E,\nabla)}Q \rightarrow T^*_{(E,\nabla)}Q$$
        Explicitly, any point in $T^*_{(E,\nabla)}Q$ may be written as $(q, p \otimes s)$, where $q \in Q$, $p \in T^*_qQ$, and $s \in E_q$. In this case,
        $$\wt{\phi}(q,p \otimes s) := (\phi(q), (\phi^{-1})^*p \otimes \Phi_q(s))$$
        (where $\Phi_q \colon E_q \xrightarrow{\sim} E_{\phi(q)}$).

        We claim that $\wt{\phi}$ is the base of an ELCS map for which $\lambda^E$ is equivariant. In fact, we claim that the ELCS map is given at the level of bundles by
        $$\wt{\Phi} := \pi^*\Phi \in \mathrm{Iso}_{\wt{\phi}}(\pi^*(E,\nabla),\pi^*(E,\nabla)).$$
        It suffices to check that $\wt{\Phi}^*\lambda^E = \lambda^E$. If $V \in T_{\wt{\phi}(q,p \otimes s)}T^*_{(E,\nabla)}M$, then by definition,
        $$\lambda^E(V) = ((\phi^{-1})^*p)(\pi_*V) \cdot \pi^*\Phi_q(s).$$
        Hence, for $W \in T_{\wt{\phi}(q,p\otimes s)}$, we have
        \begin{align*}
        	(\wt{\Phi}^*\lambda^E)_{(q,p \otimes s)}(W) &= ((\phi^{-1})^*p)(\pi_* \wt{\phi}_*W) \cdot \wt{\Phi}^*\pi^*\Phi_q(s) \\
        	&= p((\phi^{-1} \circ \pi \circ \wt{\phi})_* W)\cdot \pi^*\Phi^*\Phi_q(s) \\
        	&= p(\pi_*W) \cdot \pi^*s\\
        	&= \lambda^E_{(q,p \otimes s)}(W).
        \end{align*}
        Thus, $\wt{\Phi}^*\lambda^E = \lambda^E$.

        From the gauged perspective, with $T_{(E,\nabla)}Q\simeq (T^*Q,\pi^*\eta, \lambda^{\mathrm{can}})$ for a closed one form $\eta$ on $Q$ representing the cohomology class $[\eta] = \ln\mathrm{hol}(\nabla) \in H^1(Q;\R)$, the holonomy-preserving condition reads as $\phi^*\eta=\eta+df$ for some function $f$ determined up to addition by a constant. We obtain that $\tilde{\phi}$ is given by:

        $$\wt{\phi}(q,p) = (\phi(q),e^{-f}(\phi^*)^{-1}_qp).$$

      \end{ex}

\begin{ex}\label{excontact}
We consider the following perspective on \textbf{contact manifolds}. Let $\xi$ be a co-oriented hyperplane distribution on a manifold $M$ and let $E_\xi$ be the oriented line bundle $TM/\xi$ together with the canonical form $\alpha\in \Omega^1(M;E_\xi)$ defined by $\alpha(X)=[X]$.  Then $\xi$ is a contact structure if and only if for any connection $\nabla$ on $E_\xi$ we have that $$d_\nabla\alpha|_{\xi}:\Lambda^2\xi\rightarrow E_\xi$$ is non-degenerate. (It is sufficient to check for one connection only.)

 	Let us now fix a flat connection $\nabla$ on $E_{\xi}$. Consider $$S^+(M,\xi)=\{(q,e^t \cdot \alpha_q)| t\in \R\}\subset T^*M\otimes E.$$ This an LCS submanifold of $T_{(E,\nabla)}^*M$ (i.e. the restriction of the LCS structure is still LCS), since the result is true in the standard setting for $(E,\nabla) \cong (\underline{\R},d)$, and the LCS condition is local. Let us write $S^+(M,\xi;\nabla)$ for the smooth manifold $S^+(M,\xi)$ together with the LCS structure determined by $\nabla$, i.e. so that we may think of
 	$$S^+(M,\xi;\nabla) \subset T^*_{(E,\nabla)}M$$
 	as an LCS embedding. In fact, it is an ELCS embedding, since we may restrict the primitive $\lambda^E$.
 	
 	Recall in Example \ref{excotan} that for any holonomy preserving diffeomorphism $\phi$ of $M$, we may lift it to an LCS automorphism $\wt{\phi}$ of $T^*_{(E,\nabla)}M$ for which $\lambda^E$ is equivariant, where $\wt{\phi}$ is determined up to radially scaling in the fiber directions by a positive constant factor. In the case that we have $\phi = \mathrm{Id}$, then the lifts $\wt{\phi} = \wt{\mathrm{Id}}$ are given by scaling radially fibrewise. That is, we obtain a natural free $\R^+$-action on $T^*_{(E,\nabla)}M$ for which $\lambda^E$ is equivariant. Furthermore, $S^+(M,\xi)$ is preserved by this action. For any $T > 0$, we may quotient by the discrete subgroup $\Z \leq \R^+$ generated multiplicative by $e^T$, yielding by Proposition \ref{prop:covering_map_exact} an ELCS structure on
 	$$S^+(M,\xi) \rightarrow S^+(M,\xi)/\langle e^T \rangle =: S_T^+(M,\xi).$$
  We refer to the ELCS structrue as $S_T^+(M,\xi,\nabla)$, and refer to it as the \textbf{twisted conformal symplectisation} of $(M,\xi,\nabla)$.
  
  When $(E,\nabla)$ is trivial as a flat oriented line bundle (equivalently $\ln(\mathrm{hol}(\nabla)) = 0 \in H^1(M;\R)$), we recover the standard LC-symplectisation as discussed in \cite{zbMATH06843728} or \cite{zbMATH07103007}. When the holonomy is non trivial then this recovers the definition of twisted conformal symplectisation from \cite[Section 2.4]{allais2022dynamics}, where the Hamiltonian dynamics of these manifolds and its relation with contact dynamics is extensively studied.

	We may further generalise the twisted conformal symplectisation as follows. A \textbf{contactomorphism} between $(M_1,\xi_1)$ and $(M_2,\xi_2)$ is a diffeomorphism $\phi$ that is covered by a bundle map $\Psi \colon E_{\xi_1} \rightarrow E_{\xi_2}$ such that $\Psi^*\alpha_2=\alpha_1$. (Just like for LCS maps, if we know $\phi$ is covered by such a $\Psi$, then $\Psi$ is uniquely determined, c.f. Remark \ref{rem:base_only}.)
	
	Suppose $\phi$ is a holonomy-preserving contact automorphism of a compact contact manifold $(M,\xi)$. Notice that $\phi$ is covered by some $\Phi$, well-defined up to a positive constant, so that $\Phi^*\nabla = \nabla$. Note that $\Phi \neq \Psi$ where $\Psi$ is the bundle map with $\Psi^*\alpha = \alpha$. In general, $\Phi^*\alpha = e^{-T}\alpha$ for some $T \in C^{\infty}(M)$. Take the corresponding lift $\wt{\phi}$, which we remind the reader is given by
	$$\wt{\phi} = (\Phi^*)^{-1} \colon T^*M \otimes E \rightarrow T^*M \otimes E.$$
	But $\Phi^*$ clearly preserves $S^+(M,\xi)$, since $\Phi^*\alpha = e^{-T}\alpha$ for some $T \in C^{\infty}(M)$ and $S^+(M,\xi)$ is invariant under multiplication by positive constants.
	
	If we replace $\Phi$ with some $e^c \cdot \Phi$ for $c$ large enough, then the function $T$ increases by $c$. By choosing $c$ large enough so that $T > 0$ everywhere, the corresponding $\wt{\phi}$ generates a free action. Hence, we obtain via Proposition \ref{prop:covering_map_exact} an ELCS quotient
	$$S_T^+(M,\xi,\nabla;\phi) := S^+(M,\xi,\nabla)/\Z.$$
	This ELCS manifold is called \textbf{suspension of $(M,\xi,\nabla)$ with respect to $\phi$}. If $\phi = \mathrm{id}$ (which is holonomy-preserving with respect to any $\nabla$), we recover the twisted conformal symplectisation. 

	In the gauged perspective in which we have trivialized $(E_{\xi},\nabla) \cong (\underline{\R},d+\eta)$, we obtain a genuine contact from $\alpha \in \Omega^1(M)$ (with kernel $\xi$ such that $\alpha \colon TM \rightarrow \R$ factors through our trivialisation $TM/\xi = E_{\xi} \cong \underline{\R}$). The holonomy preserving condition is then $\phi^*\eta=\eta+df$ for some function $f \in C^{\infty}(M)$, and the fact that it is a contactomorphism reads as $\phi^*\alpha=e^g\alpha$. The contact form $\alpha$ gives an identification of $S^+(M,\xi)$ with $M\times \mathbb{R}$. To be specific, the ELCS structure is given in the gauged perspective by the triple
	$$(M \times \R, \pi^*\eta, e^t\cdot \alpha)$$
	Under this identification, the action of $\tilde{\phi}$ is given by
  $$\wt{\phi}(q,t) = (\phi(q),t+f(q)-g(q)+C)$$
  	for some constant $C$ (which may be absorbed into the function $f$). Indeed choosing $C$ big enough so that the image of $\{(q,0)\}$ does not intersect itself we obtain that this generate a free action. Diffeomorphically the quotient by this action is the mapping torus of $\phi$ but this indentification allows us to write the LCS form explicitly only when $\phi$ is a \emph{strict} contactomorphism (which means that $f-g$ is a constant function).

\end{ex}

\section{Products of LCS manifolds}
\label{sec:prod-LCS}

We apply now the construction from Section \ref{sec:twisted-product-flat} in the context of LCS manifolds and prove Theorem \ref{thm:main_LCS}.

\begin{dfn}\label{dfn:LCS_product}
Let $\scr{M}_1 = (M_1,E_1,,\nabla_1,\omega_1)$ and $\scr{M}_2 = (M_2,E_2,\nabla_2,\omega_2)$ be two LCS manifolds. Then for any isomorphism $$\Phi \colon (r_2^*E_2,r_2^*\nabla_2) \xrightarrow{\sim} (r_1^*E_1,r_1^*\nabla_1),$$ the \textbf{($\Phi$)-twisted product} of $\scr{M}_1$ and $\scr{M}_2$ is the LCS manifold
$$\scr{M}_1 \boxplus_{\Phi} \scr{M}_2 := (M_1\boxplus M_2,r_1^*E_1,r_1^*\nabla_1,\omega_1 \boxplus_\Phi (-\omega_2)).$$
\end{dfn}

For completeness, it is worth noting that the form $\omega_1 \boxplus_\Phi (-\omega_2)$ is indeed an LCS form. It is non-degenerate by pointwise linear algebra, and it is closed (in $\Omega^{2}(M_1 \boxplus M_2;r_1^*E_1),r_1^*\nabla_1)$) because $\boxplus_\Phi$ is a chain map.

This construction is natural with respect to LCS maps, as stated bellow:

\begin{thm}\label{thm:comp_property} 
Let $\scr{M}_1 := (M_1,E_1,\nabla_1,\omega_1)$, $\scr{M}'_1 := (M'_1,E'_1,\nabla'_1,\omega'_1)$, $\scr{M}_2 := (M_2,E_2,\nabla_2,\omega_2)$ and $\scr{M}'_2 := (M'_2,E'_2,\nabla'_2,\omega'_2)$ be LCS manifolds and let $(\phi_i,\Phi_i): \scr{M}_i \rightarrow \scr{M}'_i$ for $i=1,2$ be a pair of LCS maps. Then for every $\Psi \in \mathrm{Iso}(r_2^*(E_2,\nabla_2),r_1^*(E_1,\nabla_1))$, there is an LCS map
$$(\overline{\phi},\overline{\Phi}_1) \colon \scr{M}_1\boxplus_\Psi \scr{M}_2\rightarrow \scr{M}_1'\boxplus_{\overline{\Phi}_1 \circ \Psi\circ \overline{\Phi}_2^{-1}} \scr{M}'_2,$$
given uniquely up to deck transformation by the property that it covers the map $\phi_1\times \phi_2:M_1\times M_2\rightarrow M_1'\times M_2'$, and where
$$\overline{\Phi}_i := (r_i,r'_i)^*\Phi_i \in \mathrm{Iso}_{\overline{\phi}}(r_i^*(E_i,\nabla_i),(r'_i)^*(E'_i,\nabla'_i)).$$
(See Proposition \ref{prop:restrict_torsor} for the meaning of this notation.) Furthermore, if $\phi_1^t$ and $\phi_2^t$ are Hamiltonian families of such maps, then one may choose a canonical Hamiltonian family lifting $\phi_1^t \times \phi_2^t$.
  \end{thm}

  \begin{proof} Since both maps $\phi_i$ are bases of LCS map we have that $\phi_i^*(\operatorname{hol}\nabla'_i)=\operatorname{hol}\nabla_i$ and thus by the definition of $M_1'\boxplus M_2'$ the map $(r_1,r_2)\circ (\phi_1\times \phi_2)$ lifts to a map $\overline{\phi}: M_1\boxplus M_2\rightarrow M_1'\boxplus M_2'$ uniquely up to deck transformations (which covers the map $\phi_1\times \phi_2$). We claim that $(\overline{\phi},\overline{\Phi}_1)$ is our LCS map. Indeed, we have
  	\begin{align*}
  		\overline{\Phi}_1^*(\omega'_1 \boxplus_{\overline{\Phi}_1 \circ \Psi\circ \overline{\Phi}_2^{-1}} (-\omega'_2)) &= \overline{\Phi}_1^*\left[(r_1')^*\omega'_1 - (\overline{\Phi}_1 \circ \Psi\circ \overline{\Phi}_2^{-1})((r_2')^*\omega'_2)\right] \\
  		&= r_1^*\Phi_1^*\omega'_1 - \Psi \circ r_2^* \circ \Phi_2^* \omega'_2 \\
  		&= r_1^*\omega_1 - \Psi(r_2^*\omega_2) \\
  		&= \omega_1 \boxplus_{\Psi} (-\omega_2).
  		\end{align*}
  	Furthermore, if $\phi_1^t$ and $\phi_2^t$ are Hamiltonian, suppose that they are generated by the Hamiltonian vector fields $X_t$ and $Y_t$ resepctively, with Hamiltonians $G_t$ and $H_t$ respectively, so that
  	$$\iota_{X_t}\omega_1 = -\nabla G_t \qquad \mathrm{and} \qquad \iota_{Y_t}\omega_2 = -\nabla H_t.$$
  	Then by construction, we have that
  	$$(r_1)_* \frac{d}{dt}(\overline{\phi}^t) = X_t \qquad \mathrm{and} \qquad (r_2)_* \frac{d}{dt}(\overline{\phi}^t) = Y_t.$$
  	Hence,
  	\begin{align*}
  		\iota_{\frac{d}{dt}\overline{\phi}^t}(\omega_1 \boxplus_{\Psi} (-\omega_2)) &= \iota_{\frac{d}{dt}\overline{\phi}^t}\left(r_1^*\omega_1 - \Psi(r_2^*\omega_2)\right) \\
  			&= \iota_{X_t}\omega_1 - \Psi(\iota_{Y_t}\omega_2) \\
  			&= -\nabla_1 G_t + \Psi(\nabla_2 H_t) \\
  			&= (-r_1^*\nabla_1)(r_1^*G_t - \Psi(r_2^*H_t))
  	\end{align*}
  	That is, $\overline{\phi}^t$ is a Hamiltonian flow for the Hamiltonian function $r_1^*G_t - \Psi(r_2^*H_t)$.

\end{proof}
\color{black}

We are finally in a position to prove Theorem \ref{thm:main_LCS}.

\begin{proof}[Proof of Theorem \ref{thm:main_LCS}]
	We have constructed the covering map
	$$r = (r_1,r_2) \colon \scr{M}_1 \boxplus \scr{M}_2 \rightarrow M_1 \times M_2.$$
	By construction, it is a quotient of a product of covering spaces of $M_1$ and $M_2$, which are themselves only determined up to deck transformations, and so this covering map is canonical up to deck transformations. It suffices to verify the properties.
	
	\noindent (1) This follows by construction, see Definition \ref{dfn:LCS_product}, where we use the LCS structures $\omega_1 \boxplus_{\Phi} (-\omega_2)$ for $\Phi \in \Xi$.
	
	\noindent (2) The deck group of $r$ is $\mathbf{\Gamma} = \Gamma_1 \times \Gamma_2 / \Gamma_{12}$. If $(g_1,g_2) \in \mathbf{\Gamma}$, and $\Psi_{(g_1,g_2)}$ is the corresponding deck transformation, then for any $\Phi \in \Xi$,
	\begin{eqnarray*}
		\Psi^*(\omega_1 \boxplus_{\Phi} (-\omega_2)) &=& (e^{\rho_1(g_1)}\omega_1) \boxplus_{\Phi} (-e^{\rho_2(g_2)}\omega_2) \\
			&\sim& \omega_1 \boxplus_{\Phi} (-e^{-\rho_1(g_1)+\rho_2(g_2)}\omega_2) \\
			&=& \omega_1 \boxplus_{\Phi'} (-\omega_2)
	\end{eqnarray*}
	where $\sim$ is meant to represent that the identity extends to an LCS map intertwining the forms (given by multiplication on the bundle by the constant $e^{-\rho_1(g_1)}$, so that indeed $\Psi$ is an LCS map) and where in the last equality $\Phi' = e^{-\rho_1(g_1) + \rho_2(g_2)} \cdot \Phi$. Hence the deck transformations are indeed LCS maps, and in fact induce an action of $\mathbf{\Gamma}$ on $\Xi$ factoring through the $\R^+$-action via the map
	$$\mathbf{\Gamma} \rightarrow \R^+$$
	defined by $(g_1,g_2) \mapsto e^{-\rho_1(g_1) + \rho_2(g_2)}$. We see that this map is indeed injective, since if $\rho_1(g_1) = \rho_2(g_2)$, then $(g_1,g_2) \in \Gamma_{12}$ and hence represents a trivial element of $\mathbf{\Gamma}$.

	\noindent (3) The Lee classes of the LCS structures are all $r_1^*\mathfrak{l}_1 = r_2^*\mathfrak{l}_2$ by construction.
	
	\noindent (4) Let us now check for which $\Psi \in \Xi$ we have that $\Gamma(\phi)$ is an isotropic embedding. Denote by $\Phi$ the unique flat line bundle map such that the following composition is the identity as a bundle map.
	$$\xymatrix{(E_1,\nabla_1) \ar[r]^-{\Phi} \ar[d] & r_1^*(E_1,\nabla_1) \ar[r] \ar[d] & (E_1,\nabla_1) \ar[d] \\ M_1 \ar[r]^-{\Gamma(\phi)} & \scr{M}_1 \boxplus \scr{M}_2 \ar[r]^-{r_1} & M_1}$$
	Then the condition that $\Gamma(\phi)$ is isotropic with respect to some $\Psi \in \Xi$ is equivalent to the requirement that
	$$0 = \Phi^*(\omega_1 \boxplus_{\Psi} (-\omega_2)) = \Phi^*r_1^*\omega_1 - \Phi^*\Psi(r_2^*\omega_2) = \omega_1 - \Phi^*\Psi(r_2^*\omega_2).$$
	In other words, we have an isotropic embedding if and only if we have an LCS map $\scr{M}_1 \rightarrow \scr{M}_2$ given by the composition
	$$\xymatrix{(E_1,\nabla_1) \ar[d] \ar[r]^-{\Phi} & r_1^*(E_1,\nabla_1) \ar[r]^-{\Psi^{-1}} \ar[d] & r_2^*(E_2,\nabla_2) \ar[r] \ar[d] & (E_2,\nabla_2) \ar[d] \\ M_1 \ar[r]^-{\Gamma(\phi)} & \scr{M}_1 \boxplus \scr{M}_2 \ar@{=}[r] & \scr{M}_1 \boxplus \scr{M}_2 \ar[r]^-{r_2} & M_2}$$
	since this just reads $\omega_1 = \Phi^*\Psi(r_2^*\omega_2)$. But such $\Psi$ exists if and only if the bottom row is an LCS map, i.e. if $\phi = r_2 \circ \Gamma(\phi)$ is an LCS map, in which case $\Psi$ is uniquely defined.
	
	\noindent (5) and (6)
        Both of these follow from Theorem \ref{thm:comp_property}: just observe that the family of LCS map $\psi_t$ is the family  of maps given Theorem \ref{thm:comp_property} applied to $(\phi_t,\Phi_t)$ and starting at the indentity. The construction of $\Omega_{\Gamma(\phi_t)}$ in the previous point is exactly $r_1^*\omega_1\boxplus \left(-r_2^*\Phi_t^*\omega_2\right)$ which is $\Omega_{\operatorname{Id}}$. Proposition \ref{prop:LCS_vector_fields} implies that
		$$\Omega_{\Gamma(\phi_t)} = e^{\int_0^t \mu_{\tau}d\tau} \cdot \Omega_{\Gamma(\mathrm{Id})}.$$
	(recall that the $\mathbb{R}^*$-torsor is given precisely by the action of $\Phi_t$ in this situation). As $\mu=0$ in the Hamiltonian case, (6) follows also from Theorem \ref{thm:comp_property}.
\end{proof}

Finally, we prove Theorem \ref{thm:exact}:

\begin{proof}[Proof of Theorem \ref{thm:exact}]
	Suppose $\scr{M}_1$ is an ELCS manifold with primitive $\lambda_1$. The Liouville vector field $V_{\lambda_1}$ has expansion coefficient $\mu_{V_{\lambda_1}} = 1$, from which it follows that if $\phi_t$ is the flow of $V_{\lambda_1}$ and $\Phi_t$ is defined so that $(\phi_t,\Phi_t)\colon \scr{M}_1 \rightarrow \scr{M}_1$ is the corresponding LCS map. Applying Theorem \ref{thm:comp_property} with $(\phi_t,\Phi_t)$ yields an LCS syplectomorphism
	$$\scr{M}_1 \boxplus_{\Psi} \scr{M}_2 \rightarrow \scr{M}_1 \boxplus_{e^t\Psi} \scr{M}_2.$$
	If furthermore $\scr{M}_2$ is exact with chosen primitive $\lambda_2$, then
	$$\omega_1\boxplus_{\Psi}(-\omega_2)=d_{r_1^*\nabla_1}(\lambda_1\boxplus_{\Psi} (-\lambda_2))$$
	and thus $M_1\boxplus M_2$ is naturally ELCS.
\end{proof}

\begin{rem}
	Notice that for ELCS structures, we allow a change of the primitive by $\nabla$-exact forms. Indeed, we have
	$$(\lambda_1 + \nabla_1 s_1) \boxplus_{\Psi} (-\lambda_2 - \nabla_2 s_2) = \lambda_1 \boxplus_{\Psi} (-\lambda_2) + (r_1^*\nabla_1)(s_1 \boxplus_{\Psi} (-s_2)).$$
	Hence this construction of an ELCS structure on the product is well-defined in this category.
\end{rem}

\color{black}

\section{Examples.}
\label{sec:examples}

\subsection{Comparison with the contact product.}
\label{sec:comp-with-cont}

Recall that if $(M_1,\xi_1)$ and $(M_2,\xi_2)$ are contact manifolds, then there is a well-known construction of a contact product (see \cite{zbMATH06015967} for instance). Explicitly, if $\xi_1$ and $\xi_2$ have contact forms $\alpha_1$ and $\alpha_2$, then one constructs on $M_1 \times M_2 \times \R$ the contact manifold with contact form $e^t \alpha_1 + \alpha_2$.

We may naturally recover this contact product from our locally conformal symplectic picture. First, we notice that the (twisted) conformal symplectisation of a contact manifold actually recovers the contact manifold itself.

\begin{prop} \label{prop:reduce}
	Suppose $M$ is an odd-dimensional manifold, and $q \colon X \rightarrow M$ is an oriented $S^1$ fiber bundle. Suppose that there is an ELCS structure $(E,\nabla,\lambda)$ on $X$, with the property that $V_{\lambda}$ is tangent to the fibers and positively oriented. Then the following hold.
	\begin{enumerate}
		\item There is some constant $T > 0$ such that all of the fibers have the same period $T$ as orbits of $V_{\lambda}$.
		\item The fiber bundle $q$ is trivialisable, and hence admits global sections.
		\item There is a canonical cohomology class $\mathfrak{l} \in H^1(M;\R)$ such that
		$$\ln(\mathrm{hol}(\nabla)) = q^*\mathfrak{l} + T \cdot \mathrm{PD}(S),$$
		where $\mathrm{PD}(S)$ is the cohomology class Poincar\'e Dual to the homology class of a section of $q$ (and $T$ is the period of the orbits of $V_{\lambda}$).
		\item There is a canonical contact structure $\xi$ on $M$ such that for any flat connection $\overline{\nabla}$ on $E_{\xi}$ with $\ln(\mathrm{hol}(\overline{\nabla})) = \mathfrak{l}$, we have ELCS isomorphisms
		$$(X,E,\nabla,\lambda) \cong S_T^+(M,\xi,\overline{\nabla})$$
		intertwining the natural projections to $M$, and which are strict in the sense that $\lambda$ is identified with the canonical primitive $\lambda^{E_{\xi}}$ on $S_T^+(M,\xi,\overline{\nabla})$.
	\end{enumerate}
\end{prop}

\begin{proof}
  	\begin{enumerate}
  		\item It suffices to prove this result locally, so we may assume that $M = D^{2n+1}$ is an open disk. In such a case, $X$ is automatically a trivial bundle, so we may take $X = S^1 \times M = S^1 \times D^{2n+1}$, where $V_{\lambda}$ is tangent to the $S^1$-fibers over $M$. The only possible holonomy in this class is in the $S^1$ direction. Hence, in this local case of $X = S^1 \times D^{2n+1}$, we may take $(E,\nabla) = (\underline{\R},T \cdot dt)$, where $t$ is the coordinate along $S^1 = \R/\Z$, and $T \in \R$ is some constant. We thus have $\lambda \in \Omega^1(\R/\Z \times M)$, and we may write
  		$$\lambda = \alpha_t + f_t \cdot dt,$$
  		where $\alpha_t \in \Omega^1(M)$ and $f_t \in C^{\infty}(M)$. We have that $V_{\lambda} = g \cdot \partial_t$ for some $g \in C^{\infty}(\R/\Z \times M)$ with $g > 0$, and so that fact that $\lambda(V_{\lambda}) = 0$ implies $f_t = 0$. That is, we have simply
  		$$\lambda = \alpha_t.$$
  		But now, we have $\lambda = \iota_{V_{\lambda}}d_{\nabla}\lambda$, which using our explicit connection gives:
  		\begin{align*}
  			\alpha_t &= \iota_{g \partial_t}\left(d_M\alpha_t + dt \wedge (\dot{\alpha}_t + T\alpha_t)\right) \\
  				&= g \cdot (\dot{\alpha}_t + T\alpha_t)
  		\end{align*}
  		We see therefore that for each $q \in M$, we have
  		$$\frac{d}{dt}\left(e^{-\int_0^{t}\left(T - \frac{1}{g(\tau,q)}\right)d\tau}\alpha_t\right) = 0.$$
  		Thus,
  		$$\alpha_t = e^{\int_0^{t}\left(T - \frac{1}{g(\tau,q)}\right)d\tau}\beta$$
  		for some fixed $\beta$. But also, $\alpha_0 = \alpha_1$, and so we find that for all $q \in M$,
  		$$T = \int_0^1\frac{1}{g(t,q)}dt.$$
  		But the left-hand side is constant, whereas right-hand side is just the period of the orbit of $V_{\lambda}$ over $q$. Hence, all orbits have the same period (and $T > 0$).
  		\item The Gysin sequence for $q$ gives an exact sequence
  		$$0 \rightarrow H^1(M) \xrightarrow{\pi^*} H^1(X) \xrightarrow{\pi_*} H^0(M) \xrightarrow{*e} H^2(M)$$
  		where $e \in H^2(M)$ is the Euler class. The last arrow is injective unless $e = 0$. But if the last arrow is injective, then $\pi^*$ is an isomorphism, but then every class of $H^1(X)$ would have to pair with $[F]$ trivially. Hence we must have $e=0$. On the other hand, the Euler class classifies oriented $S^1$-bundles, and so it follows that $q$ is trivialisable.
  		\item Given that $X = S^1 \times M$, the result is an easy consequence of the K\"unneth Theorem together with our computation that $\langle \ln(\mathrm{hol}(\nabla)), [F] \rangle = T$.
  		\item Since $\scr{L}^{(E,\nabla)}_{V_{\lambda}}\lambda = 0$, we have that $\ker \lambda$ is preserved by the flow of $V_{\lambda}$. In addition, $V_{\lambda} \in \ker \lambda$. Thus, there is a smooth codimension-1 distribution $\xi \leq TM$ such that at any $x \in X$, we have
  		$$\ker \lambda_x = (d\pi_x)^{-1}(\xi_{q(x)}).$$
  		Notice that if $\sigma \colon M \rightarrow X$ is a section of $q$, then for any $m \in M$, we have that $\xi_m = (d\sigma)_m^{-1}(\ker \lambda_{\sigma(m)})$. In this way, working locally near $\sigma(m)$, we see that we have a hypersurface transverse to a Liouville vector field, and hence (since we are now in the standard symplectic and contact setting) we have that $(M,\xi)$ is automatically contact.

  		Pick the isomorphism of flat line bundles $\Sigma \in \mathrm{Iso}_{\sigma}(E_{\xi},E)$ such that $\Sigma^*\lambda = \alpha \in \Omega^1(M;E_{\xi})$ is the canonical contact form. Write $\overline{\nabla} = \Sigma^*\nabla$. Notice that automatically we have
  		$$\ln(\mathrm{hol}(\overline{\nabla})) = \sigma^*(\ln(\mathrm{hol}(\nabla)) = \mathfrak{l}.$$
  		We see that the Liouville flow allows us to use the section $\sigma$ to find an identification
  		$$\psi \colon X \xrightarrow{\sim} \R/T\Z \times M$$
  		where $V_{\lambda}$ is identified with $\partial_t$, and such that $\psi^{-1}(0,m) = \sigma(m)$. We similarly have a canonical section $\sigma' \colon M' \hookrightarrow S_T^+(M,\xi,\overline{\nabla})$ such that $(\sigma')^*\lambda^{E_{\xi}} = \alpha$, and which allows us to identify
  		$$\psi' \colon S_T^+(M,\xi,\overline{\nabla}) \xrightarrow{\sim} \R/T\Z \times M,$$
  		again where $V_{\lambda}$ is identified with $\partial_t$ and $(\psi')^{-1}(0,m) = \sigma'(m)$. We therefore have a diffeomorphism
  		$$\phi := \psi^{-1} \circ \psi' \colon S_T^+(M,\xi,\nabla) \rightarrow X$$
  		intertwining the projections to $M$ and the fiberwise Liouville vector fields. We see that $\phi$ preserves the Lee classes, and hence we may extend $\phi$ to an isomorphism $\Phi \in \mathrm{Iso}_{\phi}((E,\nabla),\pi^*(E_{\xi},\nabla_{\xi}))$. In fact, by Proposition \ref{prop:restrict_torsor}, we may choose $\Phi$ so that $(\sigma')^* = \Sigma^*\Phi^*$. Thus,
  		$$\Sigma^*\lambda = \alpha = (\sigma')^*\lambda^{E_{\xi}} = \Sigma^*\Phi^*\lambda^{E_{\xi}}.$$
  		It follows that $\Phi^*\lambda^{E_{\xi}}$ and $\lambda$ match on their restrictions to $\sigma(M)$. But also, both forms vanish on the Liouville vector field which is transverse to $\sigma(M)$, so actually these forms also match along the normal bundles to $\sigma(M)$. Finally, we have that both
  		$$\scr{L}^{(E,\nabla)}_{V_{\lambda}}\lambda = 0$$
  		and
  		$$\scr{L}^{(E,\nabla)}_{V_{\lambda}}\Phi^*\lambda^{E_{\xi}} = \Phi^*\scr{L}^{(E_{\xi},\nabla_{\xi})}_{V_{\lambda}}\lambda^{E_{\xi}} = \Phi^*0 = 0.$$
  		Hence, once we have that $\Phi^*\lambda^{E_{\xi}}$ and $\lambda$ match along the total tangent bundle to $X$ along $\sigma(M)$, we may flow along $V_{\lambda}$ to find that they agree everywhere. Thus, $(\phi,\Phi)$ is a strict ELCS map, i.e. with $\Phi^*\lambda^{E_{\xi}} = \lambda$.
  		
  		We have shown that we have an identification as desired for a single flat connection $\overline{\nabla}$ on $E_{\xi}$. If we have another flat connection $\overline{\nabla}'$ with the same holonomy class, then we may write $\overline{\nabla}' = \overline{\nabla} + df$ for $f \in C^{\infty}(M)$, and we obtain the base of a strict ELCS isomorphism $\phi \colon S_T^+(M,\xi,\overline{\nabla}) \cong S_T^+(M,\xi,\overline{\nabla}')$ by flowing each fiber over $m \in M$ for time $f(m)$, as is easily checked by computation. 
  	\end{enumerate}
\end{proof}

\color{black}

We may now recover the contact product. First, fix flat connections $\nabla_i$ on $E_{\xi_i} = TM_i/\xi_i$ with trivial holonomy. Recall from Example \ref{excontact} that for fixed $T > 0$, we recover the conformal symplectisations, given as ELCS manifolds
$$S_T^+(M_i,\xi_i,\nabla_i) = S^+(M_i,\xi_i,\nabla_i)/\langle e^T\rangle.$$
We see that the projection $S^+(M_i,\xi_i,\nabla_i) \rightarrow S_T^+(M_i,\xi_i,\nabla_i)$ is a covering map which trivializes the Lee class. The products
$$S_T^+(M_1,\xi_1,\nabla_1) \boxplus S_T^+(M_2,\xi_2,\nabla_2)$$
are all ELCS-symplectomorphic by Theorem \ref{thm:exact}, and by construction, they are explicitly given by the quotient
$$S^+(M_1,\xi_1,\nabla_1) \times S^+(M_2,\xi_2,\nabla_2)/\langle e^T \rangle,$$
where $e^T$ acts diagonally. The Liouville vector field on $S^+(M_1,\xi_1,\nabla_1) \times S^+(M_2,\xi_2,\nabla_2)$ is just $(V_{\lambda_1},V_{\lambda_2})$, and we see that on the quotient, this defines an $S^1$-fibration on $S_T^+(M_1,\xi_1,\nabla_1) \boxplus S_T^+(M_2,\xi_2,\nabla_2)$. Hence, applying Proposition \ref{prop:reduce}, we have that
$$S_T^+(M_1,\xi_1,\nabla_1) \boxplus S_T^+(M_2,\xi_2,\nabla_2) \cong S_T^+(Y,\xi,\nabla)$$
for some $(Y,\xi,\nabla)$, which is determined up to contactomorphism. If we fix fix sections $s_i \colon M_i \rightarrow S_T^+(M_i,\xi_i,\nabla_i)$, we obtain a codimension 1 embedding
$$M_1 \times M_2 \times \R \hookrightarrow S_T^+(M_1,\xi_1,\nabla_1) \boxplus S_T^+(M_2,\xi_2,\nabla_2)$$
given by $(x,y,t) \mapsto (e^t \cdot x, y)$, which is transverse to the Liouville vector field, and hence
$$(Y,\xi,\nabla) \cong (M_1\times M_2 \times \R, \ker \lambda|_Y).$$

\subsection{Product of twisted cotangent bundles.}
\label{sec:prod-twist-cotang}

Let $(E_1,M_1,\nabla_1)$ and $(E_2,M_2,\nabla_2)$ be two flat line bundles. Then we have an indentification of l.c.s manifolds:

$$T^*_{(E_1,\nabla_1)}M_1 \boxplus T^*_{(E_2,\nabla_2)}M_2 = T^*_{(r_1^*E_1,r_1^*\nabla_1)}(M_1 \boxplus M_2).$$

This is easily verified by Proposition \ref{prop:covering_map}, since both sides have trivializing covers $$T^*(\widetilde{M_1}\times\widetilde{M_2})\cong T^*\widetilde{M}_1\times T^*\widetilde{M}_2$$ (with their standard symplectic structures) such that the corresponding deck group actions are the same.

\subsection{Products of isotropics and Lagrangians}

Suppose that we have a submanifold $L$ of a manifold $M$ with specified class $\mathfrak{l} \in H^1(M)$. We recall from Section \ref{ssec:twisted_product} that if $q \colon \wt{M} \rightarrow M$ is the minimal trivializing cover with deck group $\Gamma$, then the class $\mathfrak{l} \in H^1(M;\R)$ may be thought of as a group homomorphism $\rho_{\mathfrak{l}} \colon H_1(M) \rightarrow \R$ which factors as
$$H_1(M) \twoheadrightarrow \Gamma \stackrel{\overline{\rho}_{\mathfrak{l}}}{\hookrightarrow} \R.$$
Similarly, the restriction of $\mathfrak{l}|_L \in H^1(L)$ corresponds to a group homomorphism $\rho_{\mathfrak{l}|_L} \colon H_1(L) \rightarrow \R$, which similarly factors as
$$H_1(L) \twoheadrightarrow \Gamma(L) \stackrel{\overline{\rho}_{\mathfrak{l}|_L}}{\hookrightarrow} \R$$
where $\Gamma(L) = H_1(L)/\ker(\rho_{\mathfrak{l}|_L})$. Since $\rho_{\mathfrak{l}|_L} = \rho_{\mathfrak{l}} \circ \iota_*$, we find that $\Gamma(L) \leq \Gamma$ is canonically a subgroup.

\begin{lemma} \label{lem:Lagn_lift}
	If $L \subset M$ is a connected submanifold, then $\Gamma(L)$ is the stabilizer of the action of $\Gamma$ on $\pi_0(q^{-1}(L))$. That is, $\pi_0(q^{-1}(L)) = \Gamma/\Gamma(L)$, and each connected component $\wt{L}$ of $q^{-1}(L)$ has that
	$$q|_{\wt{L}} \colon \wt{L} \rightarrow L$$
	is a covering map with deck group $\Gamma(L)$, i.e. the minimal trivializing cover with respect to $\mathfrak{l}|_L \in H^1(L)$.
\end{lemma}

\begin{proof}
	An element $g \in \Gamma$ is in the stabilizer of the action of $\Gamma$ on $\pi_0(q^{-1}(L))$ if and only if for any $y \in q^{-1}(L)$, there is a path $\gamma$ in $q^{-1}(L)$ starting at $y$ and ending at $g \cdot y$. This means that the loop $q \circ \gamma$, considered as a homology class in $H_1(M)$, projects to $g$ under the map $H_1(M) \twoheadrightarrow \Gamma$. But the loop $q \circ \gamma$ is contained in $L$, and hence $g$ is in the stabilizer if and only if $g \in \Gamma(L)$.
\end{proof}

Let us now consider our product construction.

\begin{thm}
	Suppose that $L_1 \subset M_1$ and $L_2 \subset M_2$ are connected smooth submanifolds, and we have specified $\mathfrak{l}_1 \in H^1(M_1)$ and $\mathfrak{l}_2 \in H^1(M_2)$. Let
	$$r \colon M_1 \boxplus M_2 \rightarrow M_1 \times M_2$$
	be our product construction in this situation, with deck group $\Gamma_{12}$. Then the stabilizer of the action of $\Gamma_{12}$ on $\pi_0(r^{-1}(L_1 \times L_2))$ is the subgroup $\Gamma_{12}(L_1,L_2) \leq \Gamma_{12}$ given by the diagram
	$$\xymatrix{\Gamma_{12}(L_1,L_2) \ar[rr] \ar[dd] \ar@{_{(}-->}[rd] & & \Gamma_1(L_1) \ar@{^{(}->}[d] \\  & \Gamma_{12} \ar[r] \ar[d] & \Gamma_1 \ar[d] \\ \Gamma_2(L_2) \ar@{^{(}->}[r] & \Gamma_2 \ar[r] & \R}$$
	where the inner and outer squares are pullback squares (so that the dotted arrow exists and is unique, and is easily checked to be an inclusion). In particular, $\pi_0(r^{-1}(L_1 \times L_2)) = \Gamma_{12}/\Gamma_{12}(L_1,L_2)$, and each connected component $\wt{L} \subset r^{-1}(L_1 \times L_2)$ satisfies that
	$$r|_{\wt{L}} \colon \wt{L} \rightarrow L_1 \times L_2$$
	is isomorphic to the covering map $L_1 \boxplus L_2 \rightarrow L_1 \times L_2$ corresponding to $\mathfrak{l}_1|_{L_1}$ and $\mathfrak{l}_2|_{L_2}$ (with deck group $\Gamma_{12}(L_1,L_2)$).
\end{thm}
\begin{proof}
	The proof is much the same as in Lemma \ref{lem:Lagn_lift}. Namely, an element $(g_1,g_2) \in \Gamma_{12}$ is seen to be in the stabilizer of the action on $\pi_0(r^{-1}(L_1 \times L_2))$ if and only if $g_1 \in \Gamma_1(L_1)$ and $g_2 \in \Gamma_2(L_2)$. That is, the stabilizer is the subgroup of $\Gamma_1 \times \Gamma_2$ given by
	$$(\Gamma_1(L_1) \times \Gamma_2(L_2)) \cap \Gamma_{12}.$$
	Since $\Gamma_1(L_1) \leq \Gamma_1$ and $\Gamma_2(L_2) \leq \Gamma_2$ (notably we have injections), the above intersection is the same as just the pullback of the outer square in our diagram, as desired.
\end{proof}

\begin{rem}
	The theorem above includes the lemma in the special case that $M_2 = * = L_2$.
\end{rem}

Finally, let us include LCS geometry into these statements.

\begin{lemma} \label{lem:product_iso}
	If $L_1 \subset (M_1,E_1,\nabla_1,\omega_1)$ and $L_2 \subset (M_2,E_2,\nabla_2,\omega_2)$ are isotropic submanifolds. Then $r^{-1}(L_1 \times L_2)$ is also an isotropic submanifold (for any of the LCS structures on $M_1 \boxplus M_2$). If both $L_1$ and $L_2$ are Lagrangian, then so is $r^{-1}(L_1 \times L_2)$.
\end{lemma}
\begin{proof}
	This is clear since $q_1^{-1}(L_1) \times q_2^{-1}(L_2) \subset \wt{M_1} \times \wt{M_2}$ is isotropic or Lagrangian.
\end{proof}

\subsection{Lagrangian suspension.}
\label{sec:lagr-susp}

As an illustration of the Lemma above we can construct stabilisation of Lagrangians submanifolds.

\begin{ex}
	The following is a slight modification of Example \ref{ex-stabilisation} to the LCS setting, where instead of taking (twisted) products with $S^1$, we take products with twisted cotangent bundles of $S^1$.

	Suppose that $\scr{M} := (M,E,\nabla,\omega)$ is an LCS manifold with Lee class $\mathfrak{l} \in H^1(M)$, and that the period subgroup is given by $\Z\alpha$ for some $\alpha \neq 0$. Meanwhile, let recall that $S^1$ comes with a line bundle $(\underline{\R},\nabla_{\alpha})$ with holonomy class $\alpha \in \R \cong H^1(S^1)$. We may thus form the product
	$$r \colon \scr{M} \boxplus T^*_{(\underline{\R},\nabla_{\alpha})}S^1 \rightarrow M \times T^*_{(\underline{\R},\nabla_{\alpha})}S^1.$$
	By Theorem \ref{thm:exact}, all of the LCS structures on this manifold are LCS-symplectomorphic since $T^*_{(\underline{\R},\nabla_{\alpha})}S^1$ is ELCS. We note that as in Example \ref{ex-stabilisation}, the deck group $\Gamma_{12}$ of $r$ is isomorphic to the deck group $\Gamma_M$ of the minimal trivializing cover for $\scr{M}$, which has $\Gamma_M \cong \Z$ since we assume that the period subgroup has rank 1, under which the map $\Gamma_M \rightarrow \R$ is just identified with $\Z \xrightarrow{\alpha} \R$. Similarly to Example \ref{ex-stabilisation}, we have that there are diffeomorphisms
	$$\scr{M} \boxplus T^*_{(\underline{\R},\nabla_{\alpha})}S^1 \cong M \times \R^2$$
	which are natural to work with from the gauged perspective.

	Suppose now that $L \subset M$ is Lagrangian. We have $\Gamma(L) \leq \Gamma \cong \Z$. Notice that we also have the Lagrangian zero section $Z = S^1 \subset T^*_{(\underline{\R},\nabla_{\alpha})}S^1$. We may thus naturally take the product of $L$ and $Z$ as described above. We obtain Lagrangian copies
	$$L \boxplus Z \subset \scr{M} \boxplus T^*_{(\underline{\R},\nabla_{\alpha})}S^1,$$
	where the number of connected components is $\Gamma_{12}/\Gamma_{12}(L,Z)$.	There are two main possibilities for what this may look like.
	\begin{itemize}
		\item In the case that $\mathfrak{l}|_L = 0 \in H^1(L)$, we have
		$$\xymatrix{ \Gamma_{12}(L,Z) \ar@{=}[r] & 0 \ar[rr] \ar[dd] \ar@{-->}[rd] & &  0\ar@{=}[r] \ar@{^{(}->}[d] & \Gamma(L) \\ & & \Z \ar@{=}[r] \ar@{=}[d] & \Z \ar@{=}[r] \ar[d]^{*\alpha} & \Gamma_M  \\ & \Z \ar@{=}[r] \ar@{=}[d] & \Z \ar[r]^{*\alpha} \ar@{=}[d] & \R & \\ & \Gamma(Z) & \Gamma_{T^*S^1} & &}$$
		(where the middle copy of $\Z$ is $\Gamma_{12}$). We see that there are $\Z$ copies of $L \boxplus Z \cong L \times \R$ embedded as Lagrangians in $\scr{M} \boxplus T^*_{(\underline{\R},\nabla_{\alpha})}S^1$, covering $L \times S^1$, and which are freely intertwined by the deck group action of $\Gamma_{12} = \Z$ of $r$.
		
		\item If $\mathfrak{l}|_L \neq 0$, then $\Gamma(L) \hookrightarrow \Gamma$ is identified with some map $\Z \hookrightarrow \Z$ given by multiplication by some integer $k$ (which we may assume is positive), and our diagram instead looks like the following:
		$$\xymatrix{ \Gamma_{12}(L,Z) \ar@{=}[r] & \Z \ar@{=}[rr] \ar[dd]_{*k} \ar@{-->}[rd]^{*k} & &  \Z \ar@{=}[r] \ar@{^{(}->}[d]^{*k} & \Gamma(L) \\ & & \Z \ar@{=}[r] \ar@{=}[d] & \Z \ar@{=}[r] \ar[d]^{*\alpha} & \Gamma_M  \\ & \Z \ar@{=}[r] \ar@{=}[d] & \Z \ar[r]^{*\alpha} \ar@{=}[d] & \R & \\ & \Gamma(Z) & \Gamma_{T^*S^1} & &}$$
		In this case, we have
		$$L \boxplus Z = (\wt{L} \times \R)/\Z,$$
		where $\wt{L}$ is the trivializing $\Gamma(L) = \Z$-cover of $L$, $\R$ is the universal cover of $Z$, and the action of $\Z$ is by $\Gamma_{12}(L,Z)$, i.e. acting on $\wt{L}$ by its identification with $\Z$ and on $\R$ by $t \mapsto t+k$. We see that we have a diffeomorphism
		$$L \boxplus Z \cong L \times \R.$$
		There are $k$ copies of these Lagrangians, intertwined by the action of $\Gamma_{12} = \Z$, but such that $k\Z$ preserves the connected components. (Under the identification of each lift $L \boxplus Z \cong L \times \R$, the action of $k\Z$ is just on the $\R$ component.)
	\end{itemize}
\end{ex}

We will call such a Lagrangian $L\times \mathbb{R}$ a \emph{stabilisation} of $L$. Just as in the symplectic case we will show how we can interpolate in $\scr{M} \boxplus T^*_{(\underline{\R},\nabla_{\alpha})}S^1 \cong M\times \mathbb{R}^2$ between stabilisations of Hamiltonian isotopic Lagrangians. For this, fix a gauge identification of the LCS structure as $(M,\eta,\omega)$, for $\eta$ a closed $1$-form and $\omega$ a non-degenerate $2$-form, as in Definition \ref{dfn:LCS_gauge}. From the discussion in Example \ref{ex-stabilisation} we have that the LCS structure on $\scr{M} \boxplus T^*_{(\underline{\R},\nabla_{\alpha})}S^1$ is given by
$$\left(M \times \R^2, \eta, \omega+e^{\alpha t}(dt+\frac{\eta}{\alpha})\wedge ds\right).$$
(Note that we have left out $\pi^*$ to refer to $\pi^*\eta$ and $\pi^*\omega$, where $\pi \colon M \times \R^2 \rightarrow M$ is the usual projection.)

For a Lagrangian embedding $i: L\hookrightarrow M$, the corresponding Lagrangian embedding of $L\times \mathbb{R}$ is always given by
$$(q,t)\mapsto (i(q),0,0).$$
Let $H:M\times \mathbb{R}^2$ be a Hamiltonian function such that $H|_{M\times (-\infty,-T)\times \mathbb{R}}(q,t,s)=0$ and $H|_{M\times (T,\infty)\times \mathbb{R}}(q,t,s)=H_0(q)$ for some Hamiltonian function on $M$. Then near $t=\infty$ the Hamiltonian vector field of $H$ is the lift of the one corresponding to $H_0$. Therefore the image of $L\times \mathbb{R}$ is a Lagrangian cylinder interpolating between $L$ and its image under the Hamiltonian isotopy given by $H$. We call this cylinder a \emph{Lagrangian suspension} of $L$ by the Hamiltonian isotopy. In the symplectic case this has been introduced studied by Chekanov \cite{zbMATH01066457}. This allows one to define an LCS notion of Lagrangian cobordism \`a la Arnol'd \cite{Arlag1,Arlag2} and a cobordism category in a way similar to the one defined and studied by Biran and Cornea \cite{zbMATH06168510,zbMATH06384095}. The study of properties of this category goes beyond the scope of the present paper and would be the subject of future work.

\color{black}

\bibliographystyle{plain}
\bibliography{Bibliographie_en}

\def\cprime{$'$} \def\cprime{$'$} \def\cprime{$'$} \def\cprime{$'$}
  \def\cprime{$'$} \def\polhk#1{\setbox0=\hbox{#1}{\ooalign{\hidewidth
  \lower1.5ex\hbox{`}\hidewidth\crcr\unhbox0}}} \def\cprime{$'$}
  \def\cprime{$'$}
\begin{thebibliography}{10}

\bibitem{allais2022dynamics}
Simon Allais and Marie-Claude Arnaud.
\newblock The dynamics of conformal hamiltonian flows: dissipativity and
  conservativity, 2022.

\bibitem{Arlag1}
V.~I. Arnol{\cprime}d.
\newblock Lagrange and {L}egendre cobordisms. {I}.
\newblock {\em Funktsional. Anal. i Prilozhen.}, 14(3):1--13, 96, 1980.

\bibitem{Arlag2}
V.~I. Arnol{\cprime}d.
\newblock Lagrange and {L}egendre cobordisms. {II}.
\newblock {\em Funktsional. Anal. i Prilozhen.}, 14(4):8--17, 95, 1980.

\bibitem{zbMATH06843728}
Giovanni Bazzoni and Juan~Carlos Marrero.
\newblock On locally conformal symplectic manifolds of the first kind.
\newblock {\em Bull. Sci. Math.}, 143:1--57, 2018.

\bibitem{zbMATH06168510}
Paul Biran and Octav Cornea.
\newblock Lagrangian cobordism. {I}.
\newblock {\em J. Am. Math. Soc.}, 26(2):295--340, 2013.

\bibitem{zbMATH06384095}
Paul Biran and Octav Cornea.
\newblock Lagrangian cobordism and {Fukaya} categories.
\newblock {\em Geom. Funct. Anal.}, 24(6):1731--1830, 2014.

\bibitem{zbMATH03782042}
Raoul Bott and Loring~W. Tu.
\newblock {\em Differential forms in algebraic topology}, volume~82 of {\em
  Grad. Texts Math.}
\newblock Springer, Cham, 1982.

\bibitem{zbMATH07103007}
Baptiste Chantraine and Emmy Murphy.
\newblock Conformal symplectic geometry of cotangent bundles.
\newblock {\em J. Symplectic Geom.}, 17(3):639--661, 2019.

\bibitem{zbMATH01066457}
Yu.~V. Chekanov.
\newblock Lagrangian embeddings and {Lagrangian} cobordism.
\newblock In {\em Topics in singularity theory. V. I. Arnold's 60th anniversary
  collection}, pages 13--23. Providence, RI: American Mathematical Society,
  1997.

\bibitem{zbMATH03951633}
Fran{\c{c}}ois Laudenbach and Jean-Claude Sikorav.
\newblock Persistance d'intersection avec la section nulle au cours d'une
  isotopie hamiltonienne dans un fibr{\'e} cotangent.
\newblock {\em Invent. Math.}, 82:349--357, 1985.

\bibitem{zbMATH06631798}
H{\^o}ng~V{\^a}n L{\^e} and Yong-Geun Oh.
\newblock Deformations of coisotropic submanifolds in locally conformal
  symplectic manifolds.
\newblock {\em Asian J. Math.}, 20(3):553--596, 2016.

\bibitem{zbMATH05652257}
Jeffrey~M. Lee.
\newblock {\em Manifolds and differential geometry}, volume 107 of {\em Grad.
  Stud. Math.}
\newblock Providence, RI: American Mathematical Society (AMS), 2009.

\bibitem{zbMATH03796853}
S.~P. Novikov.
\newblock Multivalued functions and functionals. {An} analogue of the {Morse}
  theory.
\newblock {\em Sov. Math., Dokl.}, 24:222--226, 1981.

\bibitem{OTIMAN20171}
Alexandra Otiman and Miron Stanciu.
\newblock Darboux–weinstein theorem for locally conformally symplectic
  manifolds.
\newblock {\em Journal of Geometry and Physics}, 111:1--5, 2017.

\bibitem{zbMATH06015967}
Sheila Sandon.
\newblock On iterated translated points for contactomorphisms of {{\(\mathbb
  R^{2n+1}\)}} and {{\(\mathbb R^{2n} \times S^{1}\)}}.
\newblock {\em Int. J. Math.}, 23(2):14, 2012.
\newblock Id/No 1250042.

\end{thebibliography}

\end{document}